\documentclass[12pt, reqno, a4paper]{amsart}
\usepackage[margin=1.2in]{geometry}
\numberwithin{equation}{section}
\usepackage{amssymb,amsfonts,amsthm}
\usepackage[utf8]{inputenc}
\usepackage{listings}
\usepackage{bm}
\usepackage[hyperpageref]{backref}
\usepackage{esint}
\usepackage{bigints}
\usepackage{color}
\usepackage[bookmarks]{hyperref}
\usepackage{siunitx}
\usepackage{hyperref}
\usepackage{tikz}
\usetikzlibrary{shapes}
\usepackage{caption} 
\usepackage{mathtools}

\allowdisplaybreaks[4] 

\newtheoremstyle{myremark}{10pt}{10pt}{}{}{\bfseries}{.}{.5em}{}

 \newtheorem{thm}{Theorem}
 \newtheorem{cor}{Corollary}[section]
 \newtheorem{lemma}{Lemma}[section]
 
 \theoremstyle{definition}
 \newtheorem{defn}{Definition}
 
 \newtheorem{rem}{Remark}

\usepackage{hyperref}

\allowdisplaybreaks[4]

\begin{document}

\title[Fractional Hardy inequality]{Fractional Hardy inequality with singularity on submanifold}

\author{Adimurthi, Prosenjit Roy, AND Vivek Sahu}
\address{ Department of Mathematics and Statistics,
Indian Institute of Technology Kanpur, Kanpur - 208016, Uttar Pradesh, India}

\email{Adimurthi: adiadimurthi@gmail.com}
\email{Prosenjit Roy: prosenjit@iitk.ac.in}
\email{Vivek Sahu: viiveksahu@gmail.com}



\date{}

\dedicatory{}

\begin{abstract}
We establish fractional Hardy inequality on bounded domains in $\mathbb{R}^{d}$ with inverse of distance function from smooth boundary of codimension $k$, where $k=2, \dots,d$, as weight function. The case $sp=k$ is the critical case, where optimal logarithmic corrections are required.  All the other cases of $sp<k$ and $sp>k$ are also addressed. 
\end{abstract}

\maketitle


\section{Introduction}

Hardy inequalities are fundamental tools in analysis with wide applications in mathematical physics and partial differential equations. They reflect the effect of singularities and geometry through weighted terms and have been extended to various local and nonlocal settings. In this article, we establish Hardy-type inequalities with singularities on smooth submanifolds in the nonlocal framework, where the singular weight is obtained from the distance function to the submanifold.

\smallskip

Let  $\Omega$ be a bounded domain in  $\mathbb{R}^d$, where  $d \geq 2$, and  $K$ be a compact surface with piecewise smooth boundary of codimension  $k$, where  $k = 1, \dots, d$. We adopt the convention that when  $k = 1$,  $K$ represents the boundary of  $\Omega$, and when  $k = d$,  $K$ denotes a point (say origin). For any  $x \in \Omega$, we define
\begin{equation}
    \delta_{K}(x) := \inf_{y \in  K} |x-y|.
\end{equation}

\smallskip

Let $\Omega$ be a bounded domain in $\mathbb{R}^{d}$ with Lipschitz boundary, and let $1 < p < \infty$. Consider the case where $k = 1$ and $K = \partial \Omega$. The Hardy–Sobolev inequality with singularity on the boundary of $\Omega$ for the local case (cf. \cite{lewis1988}) states that there exists a constant $C = C(d, p, \Omega) > 0$ such that
\begin{equation}\label{boundary hardy inequality}
      \int_{\Omega} \frac{|u(x)|^{p}}{\delta^{p}_{\partial \Omega}(x)} \, dx \leq C \int_{\Omega} |\nabla u(x) |^{p} \, dx, \quad \ \forall \ u \in C^{\infty}_{c}(\Omega),
\end{equation}
where  $\delta_{\partial \Omega}$ is the distance function from the boundary of  $\Omega$. Furthermore, for the local case, Barbatis, Filippas, and Tertikas in  \cite{tertikas2003} (see also \cite{Barbatis2004}) delved into the Hardy-Sobolev inequality concerning smooth boundaries with codimension  $k$, where $1 \leq k \leq d$, deriving a series expansion of Hardy-Sobolev Inequalities.

\smallskip

The scope of this article is to present fractional version of the Hardy inequalities with singularity on smooth compact sets. Chen and Song in  \cite{chen2003} studied Hardy inequality with singularity on the boundary of a domain in fractional Sobolev spaces for the case  $p=2$. Afterwards, Dyda in \cite{Dyda2004} extended the Hardy inequality from the local case to the non-local case for arbitrary $p > 1$ and established the fractional Hardy inequality with singularity on the boundary of a domain under certain boundary conditions. In particular, for a bounded Lipschitz domain  $\Omega \subset \mathbb{R}^{d}, ~ d \geq 1, ~ p>1,  ~ s \in (0,1)$ and  $sp>1$, he proved that there exists a constant  $C=C(d,p,s, \Omega) >0$, such that  
\begin{equation}\label{fractionalhardy}
    \int_{\Omega} \frac{|u(x)|^{p}}{\delta_{\partial \Omega}^{sp}(x)} \, dx \leq C \int_{\Omega} \int_{\Omega} \frac{|u(x)-u(y)|^{p}}{|x-y|^{d+sp}} \, dx \, dy, \quad \forall \ u \in C^{\infty}_{c}(\Omega). 
\end{equation} 

\smallskip

In our recent work in  \cite{adimurthi2024fractional}, we studied the case  $k=1$ and  $K= \partial \Omega$ and established the complete range of fractional Hardy inequality with singularity on the boundary of a domain considering all the cases. In the same work, we also proved appropriate optimal fractional Hardy inequalities for the case  $sp = 1$ (see   \cite[Theorem 2]{adimurthi2024fractional} and \cite{AdiPurbPro2025} for the one-dimensional case). The case $sp = 1$ turns out to be the critical case, where an additional optimal logarithmic correction is required. 

\smallskip

In this article, our objective is to extend our investigation to the fractional Hardy inequality with smooth boundaries of codimension  $k$, where $k=2, \dots, d$. This extension is motivated by the need to generalize the known results and provide a comprehensive understanding of fractional Hardy inequalities in more complex geometrical settings. By considering smooth boundaries of higher codimension, we aim to uncover new insights and establish broader results that can potentially lead to further applications and developments in the theory of fractional Sobolev spaces and their associated inequalities. To characterize these boundaries, we define a smooth boundary  $K$ of codimension  $k$ in a manner similar to the definition of a domain with a Lipschitz boundary. When we refer to a function  $f \in C^{0,1}(\Omega)$, we imply that  $f$ is Lipschitz continuous on  $\Omega$.

\begin{defn}\label{definition}
   For any  $1<k< d$ with $k \in \mathbb{N}$, let  $K$ be a compact surface of codimension  $k$ in  $\mathbb{R}^{d}$,  we say that  $K$  is a class of  $C^{0,1}$ if there exists  $C>0$ such that, for any  $x \in K$, there exists a ball  $B=B_{r}(x) \subset \mathbb{R}^{d}, ~  r >0$ and a isomorphism  $T: Q \to B$ such that
    \begin{equation*}
        T \in C^{0,1}(\overline{Q}), \hspace{3mm} T^{-1} \in C^{0,1}(\overline{B}) \hspace{3mm} \text{and} \hspace{3mm} \|T\|_{C^{0,1}(\overline{Q})} + \|T^{-1}\|_{C^{0,1}(\overline{B})} \leq C,
    \end{equation*}
    where  $Q= B^{k}_{1}(0) \times (0,1)^{d-k}$ (see Figure \ref{fig:myfigure1}), with $B^{k}_{1}(0)$ is the unit ball centered at  $0$ in $\mathbb{R}^{k}$. For any  $\xi \in Q= B^{k}_{1}(0) \times (0,1)^{d-k}$, we denote as  $\xi= (\xi_{k}, \xi_{d-k})$, where  $\xi_{k} \in B^{k}_{1}(0)$ and  $\xi_{d-k} \in (0,1)^{d-k}$. Also,  $T^{-1}(B \cap K) =  \{ (\xi_{k}, \xi_{d-k}) \in Q : \xi_{k} =0 \}$. We also assume that for any  $y \in (\Omega \setminus K) \cap B$, we have
    \begin{equation*}
        \delta_{K}(y) \sim |\xi_{k}|,
    \end{equation*}
   i.e.,  $C_{1} |\xi_{k}| \leq \delta_{K}(y) \leq C_{2} |\xi_{k}|$ for some  $C_{1}, ~ C_{2}>0$ very close to  $1$, where  $T((\xi_{k}, \xi_{d-k})) = y$. \\
\begin{figure}[h!]
\begin{center}
\begin{tikzpicture}
\node at (-1,.5) {$B^{k}_{1}(0)$};
\node at (1.5,1.5) {$(0,1)^{d-k}$};

\draw[dashed] (0,1) -- (3,1);
\draw[dashed] (0,-1) -- (3,-1);

\draw (0,0) ellipse [y radius=1, x radius=0.5];
\draw (3,0) ellipse [y radius=1, x radius=0.5];
\fill[blue!20, opacity=0.5] (0,0) ellipse [y radius=1, x radius=0.5];
\fill[blue!20, opacity=0.5] (3,0) ellipse [y radius=1, x radius=0.5];

\draw[->] (-2,0) -- (4,0) node[right] {$\mathbb{R}^{d-k}$};
\draw[->] (0,-2) -- (0,2) node[above] {$\mathbb{R}^{k}$};
\end{tikzpicture}
\end{center}
\caption{$Q= B^{k}_{1}(0) \times (0,1)^{d-k}$.}
\label{fig:myfigure1}
\end{figure}
\end{defn}

\smallskip

We also define the fractional Sobolev critical exponent for  $sp<d$ as $$ p^{*}_{s} := \frac{dp}{d-sp}.$$ Consider an open set  $\Omega$ in $\mathbb{R}^{d},~ d \geq 2, ~ s \in (0,1)$, and let  $K$ be a compact subset of  $\Omega$, we define the space  $W^{s,p}_{0}(\Omega \setminus K)$, consisting of closure of  $W^{s,p} (\Omega) \cap C(\Omega)$ functions (see  \eqref{defn Wsp} for the definition of  $W^{s,p}(\Omega)$) that vanishes in a neighbourhood of $K$ with respect to the norm  $\| \cdot \|_{W^{s,p}(\Omega)}$ (see  \eqref{defn norm} for the definition of  $\| \cdot \|_{W^{s,p}(\Omega)}$), 
\begin{equation*}
    W^{s,p}_{0}(\Omega \setminus K) := \text{closure} \left( \{ u \in W^{s,p} (\Omega) \cap C(\Omega) : u=0 \ \text{in a neighbourhood of K} \} \right).
\end{equation*}

\smallskip

The following theorem is our main result. It establishes the complete range of fractional Hardy inequalities for the case $sp = k$ on a bounded Lipschitz domain $\Omega \subset \mathbb{R}^{d}$, $d \ge 2$, with a smooth boundary $K$ of codimension $k$ of class $C^{0,1}$, where $k = 2, \dots, d$, together with an optimal logarithmic weight function. Throughout this article, we assume $s \in (0,1)$ and $p > 1$, and in the case $k = d$, we assume $0 \in \Omega$ and $K = \{0 \}$. We prove the following theorem for the case $sp = k$:

\begin{thm}\label{theorem 1 sp=k}
  Let $\Omega$ be a bounded Lipschitz domain in $\mathbb{R}^{d}$ with $d \geq 2$, and let $K \subset \Omega$ be a compact set of codimension $k$ of class $C^{0,1}$, where $1 < k \leq d$ and $k \in \mathbb{N}$. Assume that $sp = k$, $\delta_{K}(x) < R$ for all $x \in \Omega$, for some $R > 0$, and $\alpha = d + (k - d) \frac{\tau}{p}$. When $k = d$, we assume that $0 \in \Omega$ and $K = \{ 0 \}$. Then, the following fractional Hardy-type inequality
    \begin{multline}
     \left(   \bigintsss_{\Omega} \frac{|u(x)|^{\tau}}{\delta^{\alpha}_{K}(x) \ln^{\beta} \left( \frac{2R}{\delta_{K}(x)}  \right)} \, dx  \right)^{\frac{1}{\tau}} \leq C \left( \int_{\Omega} \int_{\Omega}  \frac{|u(x)-u(y)|^{p}}{|x-y|^{d+sp}} \, dx \, dy + \int_{\Omega} |u(x)|^{p} \, dx \right)^{\frac{1}{p}}, \\ \forall \ u \in W^{s,p}_{0}(\Omega \setminus K),
    \end{multline}
  where  $C=C(d,p,s, \tau, R, K,\Omega)>0$,  holds true in each of the following cases: 
    \begin{itemize}
        \item[(1)] if  $1<k<d$ and  $\tau =p$, then  $\beta =p$.
        \item[(2)] if  $1<k<d$ and  $\tau \in \left(p, \frac{dp+1}{1-(k-d)} \right]$, then  $\beta = \tau -1$.
        \item[(3)] if  $1<k<d$ and  $\tau \in \left( \frac{dp+1}{1-(k-d)}, p^{*}_{s} \right]$, then  $\beta = dp + (k-d)\tau$.
        \item[(4)] if  $k=d$ and  $\tau \geq p$, then  $\beta= \tau$.
    \end{itemize}
   The weight function is optimal for the case  $\tau =p$.
\end{thm}

\smallskip

The aforementioned theorem can be equivalently expressed as 
\begin{equation*}
  \left(  \bigintsss_{\Omega} \frac{|u(x)-(u)_{\Omega}|^{\tau}}{\delta^{\alpha}_{K}(x) \ln^{\beta} \left( \frac{2R}{\delta_{K}(x)}  \right)} \, dx \right)^{\frac{1}{\tau}} \leq C \left( \int_{\Omega} \int_{\Omega}  \frac{|u(x)-u(y)|^{p}}{|x-y|^{d+sp}} \, dx \, dy \right)^{\frac{1}{p}} , \quad \forall \ u \in W^{s,p}_{0}(\Omega \setminus K),
\end{equation*}
where  $(u)_{\Omega}$ denotes the average of  $u$ over  $\Omega$ and defined by  $(u)_{\Omega} = \frac{1}{|\Omega|} \int_{\Omega} u(y) dy$. The above fractional Hardy-type inequality is obtained by using the above theorem together with the fractional Poincar\'e inequality (see \eqref{poincare}). When $sp=k$, we have $W^{s,p}_{0}(\Omega \setminus K)= W^{s,p}(\Omega)$ for a smooth compact set $K$ of codimension $k$ (see Lemma  \ref{density lemma}). Therefore, the presence of the optimal function  $\frac{1}{\ln^{p}}$ (for the case  $\tau=p$) on the left-hand side and $L^{p}$ term on the right-hand side in Theorem  \ref{theorem 1 sp=k} are crucial since constant functions are in $W^{s,p}_{0}(\Omega \setminus K)$ for the case  $sp=k$, and the function  $\frac{1}{\delta^{k}_{K}(x)}$ (for the case  $\tau=p$) is not integrable near $K$. Concluding this, we present the following remark:

\begin{rem}
    \noindent We prove that  $W_{0}^{s,p}(\Omega \setminus K)= W^{s,p}(\Omega)$ for the case  $sp=k$, where $1<k<d, ~ k \in \mathbb{N}$ (see Lemma  \ref{density lemma}). Therefore, the additional  $L^{p}$ term on the right-hand side of Theorem  \ref{theorem 1 sp=k} and  the integrability of  $\frac{1}{\delta^{\alpha}_{K}(x) \ln^{\beta} \left( \frac{2R}{\delta_{K}(x)}  \right)}$ are crucial for establishing the appropriate fractional Hardy inequality, as it accounts for the presence of constant functions in  $W^{s,p}_{0}(\Omega \setminus K)$ for the case  $sp=k$. 
\end{rem}

\begin{rem}
    \noindent For $1<k<d, ~  k \in \mathbb{N}$, and for any $\tau \in \left(p, \frac{dp+1}{1-(k-d)} \right]$, we improved the logarithmic weight function presented in the Theorem  \ref{theorem 1 sp=k} with $\beta = \tau -1$ whereas for $\tau=p$, it is $\beta=p$. Furthermore, for any $\tau \in \left( \frac{dp+1}{1-(k-d)}, p^{*}_{s} \right]$, the logarithmic weight function results from interpolation and the fractional Sobolev inequality which is done in Lemma  \ref{flat case sp=k 2} for the flat boundary case. For general bounded Lipschitz domain with smooth compact set $K$ follows from patching argument.
\end{rem}

\smallskip

Let  $K_{1}$ be any set such that  $K_{1} \subset K$, where  $K$ is a compact set of codimension  $k$ of class  $C^{0,1}$. If we consider Theorem  \ref{theorem 1 sp=k} with this compact set  $K$, we obtain a fractional Hardy inequality with singularity on  $K_{1}$ and  $sp=k$.  Therefore, we present the following corollary which is a consequence of Theorem  \ref{theorem 1 sp=k}:

\begin{cor}\label{cor 1}
    Let  $\Omega$ be a bounded Lipschitz domain in $\mathbb{R}^{d}$ with $  d \geq 2$, and let $K_{1} \subset K$ be any set such that  $K \subset \Omega$, where  $K$ a compact set of codimension  $k$ of class  $C^{0,1}$,  $1<k< d, ~  k \in \mathbb{N}$. Assume that  $sp=k$,  $\delta_{K}(x) < R$ for all  $x \in \Omega$, for some  $R>0$, and  $\alpha= d + (k-d)\frac{\tau}{p}$. Then, the following fractional Hardy-type inequality
    \begin{multline}
     \left(   \bigintsss_{\Omega} \frac{|u(x)|^{\tau}}{\delta^{\alpha}_{K_{1}}(x) \ln^{\beta} \left( \frac{2R}{\delta_{K_{1}}(x)}  \right)} \, dx \right)^{\frac{1}{\tau}} \leq C\left( \int_{\Omega} \int_{\Omega}  \frac{|u(x)-u(y)|^{p}}{|x-y|^{d+sp}} \, dx \, dy + \int_{\Omega} |u(x)|^{p} \, dx \right)^{\frac{1}{p}}, \\ \forall \ u \in W^{s,p}_{0}(\Omega \setminus K),
    \end{multline}
    holds true in each of the following cases: 
    \begin{itemize}
        \item[(1)] if  $\tau =p$, then  $\beta =p$.
        \item[(2)] if  $\tau \in \left(p, \frac{dp+1}{1-(k-d)} \right]$, then  $\beta = \tau -1$.
        \item[(2)] if  $\tau \in \left( \frac{dp+1}{1-(k-d)}, p^{*}_{s} \right]$, then  $\beta = dp + (k-d)\tau$.
    \end{itemize}
\end{cor}

\smallskip

The proof of the above corollary, using Theorem \ref{theorem 1 sp=k}, is provided in Subsection \ref{proof of cor 1}. The main idea is to obtain an inequality $\delta^{\alpha}_{K_{1}}(x) \ln^{\beta} \left( \frac{2R}{\delta_{K_{1}}(x)} \right) \geq \delta^{\alpha}_{K}(x) \ln^{\beta} \left( \frac{2R}{\delta_{K}(x)} \right)$ in a small neighbourhood of $K$, and then using Theorem  \ref{theorem 1 sp=k} and the obvious fractional Sobolev inequality, the above corollary is established.

\smallskip

Furthermore, if  $K$ be any set in  $\mathbb{R}^{d}$ such that  $K_{1}$ and  $K_{2}$ are two compact sets in  $\mathbb{R}^{d}$ of codimension  $k_{1}$ and  $k_{2}$ of class  $C^{0,1}$ respectively with  $k_{1} <  k_{2}$ and  $K \subset K_{1} \cup K_{2}$. Assume,  $\gamma = \theta k_{1} + (1- \theta) k_{2}$ for some $\theta \in (0,1)$. Additionally, we also assume that in a small neighbourhood of  $K_{1} \cup K_{2}$ there exists a constant  $C>0$ such that
\begin{equation*}
    \delta^{\gamma}_{K} \ln^{p} \left( \frac{2R}{\delta_{K}(x)} \right)  \geq C \delta^{\theta k_{1}}_{K_{1}} \ln^{\theta p} \left( \frac{2R}{\delta_{K_{1}}(x)} \right) \delta^{(1-\theta) k_{2}}_{K_{2}} \ln^{(1-\theta) p} \left( \frac{2R}{\delta_{K_{2}}(x)} \right),
\end{equation*}
where  $\max \{ \delta_{K_{1}}(x), \delta_{K_{2}}(x) \} < R$ for all  $x \in \Omega$, for some  $R>0$. Then using H$\ddot{\text{o}}$lder's inequality and Theorem  \ref{theorem 1 sp=k} for the compact set  $K_{1}$ with  $s_{1}p = k_{1}$  and  $K_{2}$ with  $s_{2}p=k_{2}$, we again present the corollary that follows from Theorem  \ref{theorem 1 sp=k}. In this corollary, we consider the norm $\| \cdot \|_{W^{s,p}(\Omega)}$ on the fractional Sobolev space, as defined in  \eqref{defn norm}.

\begin{cor}\label{cor 2}
      Let  $\Omega$ be a bounded Lipschitz domain in $\mathbb{R}^{d}$ with $  d \geq 2$,  and let $K \subset K_{1} \cup K_{2}$ be any set such that  $K_{1} \cup K_{2} \subset  \Omega$, where  $K_{1},   K_{2}$ be compact sets of codimension  $k_{1},   k_{2}$ of class  $C^{0,1} $ respectively such that  $k_{1}<k_{2}$,  $k_{1},  k_{2} \in \{2, \dots, d-1 \}$. Let  $sp=\gamma$ such that  $\gamma =\theta k_{1}+ (1- \theta)k_{2}$ for some  $\theta \in (0,1)$ and  $s= \theta s_{1} + (1- \theta)s_{2}$ for some  $s_{1},   s_{2} \in (0,1)$. Assume that  $\max \{ \delta_{K_{1}}(x), \delta_{K_{2}}(x) \} < R$ for all  $x \in \Omega$, for some  $R>0$. Then there exists a constant  $C>0$ such that
    \begin{multline}
       \left( \bigintsss_{\Omega} \frac{|u(x)|^{p}}{\delta^{\gamma}_{K}(x) \ln^{p} \left( \frac{2R}{\delta_{K}(x)}  \right)} \, dx \right)^{\frac{1}{p}} \leq C \|u\|^{\theta}_{W^{s_{1},p}(\Omega)} \|u\|^{1-\theta}_{W^{s_{2},p}(\Omega)} , \\ \forall \ u \in W^{s_{1},p}_{0}(\Omega \setminus K_{1}) \cap W^{s_{2},p}_{0} (\Omega \setminus K_{2}).
    \end{multline}
\end{cor}

\smallskip

The next two results establish the fractional Hardy inequalities with smooth boundary $K$ of codimension  $k$ of clas  $C^{0,1}$ for the case  $sp \neq k$. The following theorem establishes the fractional Hardy inequality for the case  $sp<k$:

\begin{thm}\label{theorem 2 sp<k}
    Let  $\Omega$ be a bounded Lipschitz domain in $\mathbb{R}^{d}$  with $d \geq 3$, and let  $K \subset \Omega$ be a compact set of codimension  $k$ of class  $C^{0,1}, ~  1<k<d,  ~ k \in \mathbb{N}$,  $sp<k$, and  $\alpha= d + (sp-d)\frac{\tau}{p}$. Then, for any  $\tau \in [p, p^{*}_{s}]$, there exists a constant  $C= C(d, p,s, \tau,k,K,\Omega)>0$ such that 
    \begin{multline}
     \left(   \bigintssss_{\Omega} \frac{|u(x)|^{\tau}}{\delta^{\alpha}_{K}(x)} \, dx \right)^{\frac{1}{\tau}} \leq C \left( \int_{\Omega} \int_{\Omega}  \frac{|u(x)-u(y)|^{p}}{|x-y|^{d+sp}} \, dx \, dy + \int_{\Omega} |u(x)|^{p} \, dx \right)^{\frac{1}{p}} , \\ \forall \ u \in W^{s,p}_{0}(\Omega \setminus K).
    \end{multline}
\end{thm}

\smallskip

We also establish that for any  $p>k$,  $W^{s,p}_{0}(\Omega \setminus K) = W^{s,p}(\Omega)$ when  $sp<k$ in Remark  \ref{rem sp<k}. Therefore, the above theorem does not hold true for constant functions in  $W^{s,p}_{0}(\Omega \setminus K)$ whenever  $p>k$ and  $sp<k$ unless the additional  $L^{p}$ term is included on the right-hand side. Therefore, the presence of extra  $L^{p}$ term on the right-hand side of above theorem is essential for establishing our main result in the case  $sp<k$.

\smallskip

The next theorem establishes a similar type of fractional Hardy inequality for the case  $sp >k$. In general, our next theorem is the following: 

\begin{thm}\label{theorem 3 sp>k}
     Let  $\Omega$ be a bounded Lipschitz domain in $\mathbb{R}^{d}$ with  $d \geq 3$, and let  $K \subset \Omega$ be a compact set of codimension  $k$ of class  $C^{0,1}$, $1<k<d, ~  k \in \mathbb{N}$, and  $sp>k$. Then, the following fractional Hardy-type inequality
    \begin{equation}
       \left( \bigintssss_{\Omega} \frac{|u(x)|^{\tau}}{\delta^{\alpha}_{K}(x)} \, dx \right)^{\frac{1}{\tau}} \leq C \left( \int_{\Omega} \int_{\Omega}  \frac{|u(x)-u(y)|^{p}}{|x-y|^{d+sp}} \,  dx \, dy \right)^{\frac{1}{p}} , \quad \forall \ u \in W^{s,p}_{0}(\Omega \setminus K),
    \end{equation}
   where  $C=C(d, p,s, \tau,k, K,\Omega)>0$, holds true in each of the following cases:
    \begin{itemize}
        \item[(1)] if  $sp<d$ and  $\tau \in [p, p^{*}_{s}]$, then  $\alpha = d+ (sp-d) \frac{\tau}{p}$.
        \item[(2)] if  $sp=d$ and  $\tau \geq p$, then  $\alpha = sp$.
        \item[(3)] if  $sp>d$ and  $\tau =p$, then  $\alpha = sp$.
        \end{itemize}
\end{thm}

\smallskip

Theorem  \ref{theorem 2 sp<k} and Theorem  \ref{theorem 3 sp>k} generalise fractional Hardy inequality with singularity on compact set  $K$ of codimension  $k$ of class $C^{0,1}$ for suitable  $\tau$ depending on  $s$ and  $p$. However, the generalisation is straightforward for the case  $sp<d$. It suffices to establish the case  $\tau=p$. For suitable  $\tau$, this follows from interpolation and fractional Sobolev inequality, i.e.,  H$\ddot{\text{o}}$lder's inequality, fractional Hardy inequality with singularity on smooth compact set  $K$ for the case  $\tau=p$ and fractional Sobolev inequality which is done in Subsection  \ref{proof of theorem 3} in the proof of Theorem  \ref{theorem 3 sp>k}.

\begin{rem}
    \noindent Corollary  \ref{cor 1} and Corollary  \ref{cor 2} are obtained from Theorem  \ref{theorem 1 sp=k}. Analogous corollaries can be established for the case  $sp \neq k$, where $ 1<k<d, ~ k \in \mathbb{N}$, by applying Theorem  \ref{theorem 2 sp<k} and Theorem  \ref{theorem 3 sp>k}.
\end{rem}

In our main results, we consider the case  $K \subset \Omega$, where  $\Omega$ is a bounded Lipschitz domain and $K$ is a compact set of codimension  $k$ of class  $C^{0,1}$. Using similar techniques, these results can also be extended to the case  $K \subset \overline{\Omega}$ with suitable choice of  $K$.

\smallskip

Several related developments on Hardy, Hardy–Rellich, and Poincar\'e-type identities across different geometric settings, such as Carnot groups, Bessel-pair setups, upper half-spaces, hyperbolic spaces, Dunkl operators, and mean-distance Hardy inequalities, are available in \cite{Berchio2022,Nguyen2022,Flynn2021,Flynn2025,Ghoussoub2011,Lam2019,Nguyen2025,Wang2022} and the references therein.

\subsection*{Application:} Fractional Hardy inequalities with singularities on submanifolds have applications in studying regional fractional Laplacian weighted problems. For example, consider the following fractional $p$-Laplacian weighted problem:
\begin{equation}\label{problem}
    \begin{dcases}
   (- \Delta_{p})^{s}u(x) = \lambda \frac{|u(x)|^{p-2}u(x)}{\delta^{sp}_{K}(x)}, & x \in \Omega \setminus K \\
        u(x)=0, & x \in  \mathbb{R}^{d} \setminus (\Omega \setminus K),
    \end{dcases}
\end{equation}
where $\lambda>0$. Here, $\Omega \subset \mathbb{R}^{d}$ is a bounded Lipschitz domain and $K \subset \Omega$ is a compact set of codimension $k$ of class $C^{0,1}$. Assuming $k<sp<d$ and applying Theorem \ref{theorem 3 sp>k}, it would be interesting to study the above problem.

\smallskip

For additional literature on Hardy-type inequalities, we refer to  \cite{adi2002, adi2009, tertikas2003, debdip2020, Brezis1997, Brezis19971, sandeep2008, delpino, lu2022, lu20223, lulam, Nguyen2019} and to the works mentioned there in.
 For the recent development of fractional Hardy-type inequalities and other related inequalities, we refer to the work of  \cite{Adi2025, csato, Cabre2022, lu2, lu2004, dyda2022sharp, frank2008,  Frank2010, squassina2018}.

\smallskip

The article is organized in the following way: In Section  \ref{preliminaries}, we present preliminary lemmas and notation that will be utilized to prove our main results. In section  \ref{The flat boundary case}, we prove our main results for the flat boundary case. We assume the domain $\Omega= B^{k}_{1}(0) \times (-n,n)^{d-k}$, where $B^{k}_{1}(0)$ is a unit ball with center $0$ in $\mathbb{R}^{k}, ~  n \in \mathbb{N}$ and the compact set $K=\{ (x_{k}, x_{d-k}) : x_{k}=0 \ \text{and} \ x_{d-k} \in [-n,n]^{d-k} \}$. We then prove our main results with this $\Omega$ and compact set $K$. Section  \ref{proof of main result} contains the proof of our main results for general bounded Lipschitz domain with a smooth compact set. In section  \ref{optimality section}, we present the proof of optimality of weight function for Theorem  \ref{theorem 1 sp=k}. 

\section{Notation and Preliminaries}\label{preliminaries} 
In this section, we introduce the notation, definitions, and preliminary results that will be used throughout the article. Throughout this article, we shall use the following notation: 
\begin{itemize}
\item for a measurable set  $\Omega \subset \mathbb{R}^d, ~ (u)_{\Omega}$ will denote the average of the function  $u$ over  $\Omega$, i.e., 
\begin{equation*}
    (u)_{\Omega} :=  \frac{1}{|\Omega|} \int_{\Omega} u(x) \,  dx = \fint_{\Omega} u(x) \, dx. 
\end{equation*}
Here,  $|\Omega|$ denotes the Lebesgue measure of  $\Omega$.
\item we denote by  $B^{k}_{r}(x)$, an open ball of radius  $r$ and center  $x$ in  $\mathbb{R}^{k}$, and for simplicity we denote $B_{r}(x)$, an open ball in $\mathbb{R}^{d}$.  $\mathbb{S}^{d-1}$ denotes the boundary of the open ball of radius $1$ centered at $0$ in  $\mathbb{R}^{d}$. 
\item the parameter  $s$ will always be understood to be in  $(0,1)$.
\item any point  $x \in \mathbb{R}^d$ is represented as  $x=(x_{k},x_{d-k})$, where  $x_{k} \in \mathbb{R}^{k}$ and  $x_{d-k} \in \mathbb{R}^{d-k}$.
\item for any  $f, ~ g: \Omega ( \subset \mathbb{R}^d) \to \mathbb{R}$, we denote  $f \sim g$ if there exists  $C_{1}, ~ C_{2}>0$ such that  $C_{1}g(x) \leq f(x) \leq C_{2} g(x)$ for all  $x \in \Omega$.
\item  $C>0$ will denote a generic constant that may change from line to line.
\end{itemize}

\smallskip

Let  $\Omega$ be an open set in $\mathbb{R}^{d}$ and  $s \in (0,1)$. For any  $p \in [1, \infty)$, define the fractional Sobolev space
\begin{equation}\label{defn Wsp}
W^{s, p}(\Omega) := \left\{    u \in L^{p}(\Omega) : \int_{\Omega} \int_{\Omega}  \frac{|u(x)-u(y)|^{p}}{|x-y|^{d+sp}} \, dx \, dy < \infty \right\},
\end{equation}
endowed with the norm
\begin{equation}\label{defn norm}
    \|u\|_{W^{s, p}(\Omega)} = \left(  [u]^{p}_{W^{s, p}(\Omega)} + \|u\|^{p}_{L^{p}(\Omega)}  \right)^{\frac{1}{p}},
\end{equation}
where 
\begin{equation*}
    [u]_{W^{s, p}(\Omega)} := \left( \int_{\Omega} \int_{\Omega}  \frac{|u(x)-u(y)|^{p}}{|x-y|^{d+sp}} \,  dx \, dy \right)^{\frac{1}{p}} 
\end{equation*}
is the so-called Gagliardo seminorm. Denote by  $W^{s, p}_{0}(\Omega)$, the closure of  $C^{\infty}_{c}(\Omega)$ with respect to the norm  $\| \cdot \|_{W^{s, p}(\Omega)}$.
\smallskip

One has, for any  $a_{1},  \dots , a_{m} \in \mathbb{R}$ and  $\gamma \geq1$,
\begin{equation}\label{sumineq}
    \sum_{\ell=1}^{m} |a_{\ell}|^{\gamma} \leq  \left( \sum_{\ell=1}^{m} |a_{\ell}| \right)^{\gamma} .
\end{equation}
\smallskip

{\bf Fractional Poincar\'e inequality}  \cite[Theorem 3.9]{edmunds2022}: Let  $p \geq 1$ and  $\Omega$ be a bounded open set in $\mathbb{R}^{d}$. Then there exists a constant  $C = C(d,p,s, \Omega)>0$ such that
\begin{equation}\label{poincare}
    \int_{\Omega} |u(x)-(u)_{\Omega}|^{p} \, dx \leq C [u]^{p}_{W^{s, p}(\Omega)} .
\end{equation} 

\smallskip

The following lemma establishes the fractional Sobolev inequality for bounded Lipschitz domains  $\Omega$ that scale with the parameter  $\lambda>0$. This lemma plays a crucial role in the proof of our main results

\begin{lemma}\label{sobolev} 
   Let  $\Omega$ be a bounded Lipschitz domain in  $\mathbb{R}^{d}$, where $d \geq 1$. Let  $p > 1$ and  $s \in (0,1)$ be such that  $sp \leq d$. Define  $\Omega_{\lambda}: = \left\{ \lambda x : ~ x \in \Omega \right\}$ for  $\lambda>0$, then there exists a positive constant  $C=C(d,p,s, \tau, \Omega)$ such that, for any  $u \in W^{s, p}(\Omega_{\lambda})$, we have
    \begin{equation}
         \left( \fint_{\Omega_{\lambda}} |u(x)-(u)_{\Omega_{\lambda}}|^{\tau } \, dx \right)^{\frac{1}{\tau}}  \leq C \left( \lambda^{sp-d} [u]^{p}_{W^{s, p}(\Omega_{\lambda})} \right)^{\frac{1}{p}}  ,
    \end{equation}
    for any  $\tau \in [p, p^{*}_{s}]$ when  $sp<d$, and  $\tau \geq p$ when  $sp=d$.
\end{lemma} 
\begin{proof}
  Let  $\Omega$ be a bounded Lipschitz domain in $\mathbb{R}^{d}$. Then from Theorem  $6.7$ of  \cite{di2012hitchhikers}, for any  $\tau \in [p,p^{*}_{s}]$ if  $sp<d$, and from Theorem  $6.10$ of  \cite{di2012hitchhikers}, for any  $\tau \geq p$ if  $sp=d$, we have
\begin{equation}\label{Sobineqq}
      \|u\|_{L^{\tau}(\Omega)} \leq C \|u\|_{W^{s, p}(\Omega)},
  \end{equation}
   where  $C=C(d,p,s,\tau, \Omega)$ is a positive constant. Applying  \eqref{Sobineqq} with  $u-(u)_{\Omega}$ and using  \eqref{poincare}, we have
    \begin{equation*}
        \|u-(u)_{\Omega}\|_{L^{\tau}(\Omega)} \leq C [u]_{W^{s, p}(\Omega)} .
    \end{equation*}
    Therefore,
    \begin{equation}\label{ineqlemma}
       \left(  \fint_{\Omega} |u(x)-(u)_{\Omega}|^{\tau } \, dx \right)^{\frac{1}{\tau}}  \leq C  [u]_{W^{s, p}(\Omega)}  . 
    \end{equation}
    Let us apply the above inequality to  $u(\lambda x)$ instead of  $u(x)$. This gives
    \begin{equation*}
        \left(  \fint_{\Omega}  \Big|u(\lambda x)-\fint_{\Omega} u(\lambda x) \, dx \Big|^{\tau } \, dx \right)^{\frac{1}{\tau}}  \leq C \left( \int_{\Omega} \int_{\Omega} \frac{|u(\lambda x) - u(\lambda y)|^{p}}{|x-y|^{d+sp}} \, dx \, dy  \right)^{\frac{1}{p}}  . 
    \end{equation*}
    Using the fact
    \begin{equation*}
        \fint_{\Omega} u(\lambda x) \, dx = \fint_{\Omega_{\lambda}} u(x) \, dx,
    \end{equation*}
    we have
    \begin{equation*}
       \left(  \fint_{\Omega} |u(\lambda x)-(u)_{\Omega_{\lambda}}|^{\tau } \, dx \right)^{\frac{1}{\tau}}  \leq C \left( \int_{\Omega} \int_{\Omega} \frac{|u(\lambda x) - u(\lambda y)|^{p}}{|x-y|^{d+sp}} \, dx \, dy  \right)^{\frac{1}{p}}  . 
    \end{equation*}
    By changing the variables  $X=\lambda x$ and  $Y= \lambda y$, we obtain 
    \begin{equation*}
         \left( \fint_{\Omega_{\lambda}} |u(x)-(u)_{\Omega_{\lambda}}|^{\tau } \, dx \right)^{\frac{1}{\tau}}  \leq C \left( \lambda^{sp-d} [u]^{p}_{W^{s, p}(\Omega_{\lambda})} \right)^{\frac{1}{p}}  . 
    \end{equation*}
    This finishes the proof of the lemma.
\end{proof}

\smallskip

The next lemma provides a relation between average of $u$ over two disjoint sets. This technical step is very crucial in the proof of the main results.

\begin{lemma}\label{avg}
    Let  $E$ and  $F$ be disjoint sets in  $\mathbb{R}^d$. Then for any  $\tau \geq 1$, one has for some constant  $C >0$ such that
    \begin{equation}
        |(u)_{E} - (u)_{F}|^{\tau} \leq C \frac{|E \cup F|}{\min \{ |E|, |F| \} }  \fint_{E \cup F} |u(x)-(u)_{E \cup F}|^{\tau} \, dx  . 
    \end{equation}
\end{lemma}
\begin{proof}
Let us consider  $|(u)_{E}-(u)_{F}|^{\tau}$, we have
\begin{align*}
    |(u)_{E}-(u)_{F}|^{\tau} & = |(u)_{E} - (u)_{E \cup F} - (u)_{F} + (u)_{E \cup F}|^{\tau} \\ &
        \leq C|(u)_{E} - (u)_{E \cup F}|^{\tau} +  C | (u)_{F} + (u)_{E \cup F}|^{\tau} \\ &
        = C \Big| \fint_{E} \{u(x) - (u)_{E \cup F} \} \,  dx  \Big|^{\tau} + C \Big| \fint_{F} \{ u(x) - (u)_{E \cup F} \} \,  dx \Big|^{\tau}.
\end{align*}
    By using H$\ddot{\text{o}}$lder's inequality with  $ \frac{1}{\tau} +  \frac{1}{\tau'} = 1$, we have
    \begin{align*}
         |(u)_{E}-(u)_{F}|^{\tau} & \leq  C \fint_{E} |u(x) - (u)_{E \cup F} |^{\tau} \,  dx + C \fint_{F} | u(x) - (u)_{E \cup F} |^{\tau} \, dx \\ &
       \leq  \frac{C}{\min \{ |E|, |F| \} } \int_{E \cup F} |u(x) - (u)_{E \cup F} |^{\tau} \,  dx  \\ &
       = C \frac{|E \cup F|}{\min \{ |E|, |F| \}} \fint_{E \cup F} |u(x) - (u)_{E \cup F} |^{\tau} \,  dx .
    \end{align*}
    This finishes the proof of the lemma.
\end{proof}
\smallskip

The next lemma establishes an inequality when any function  $u \in W^{s, p}(\Omega)$ is multiplied by a test function. This lemma plays a crucial role in establishing our main results.
\begin{lemma}\label{testfunc}
    Let  $\Omega$ be an open set in  $\mathbb{R}^d$. Let us consider  $u \in W^{s, p}(\Omega)$ and  $\xi \in C^{0,1}(\Omega), ~ 0 \leq \xi \leq 1$. Then  $\xi u \in W^{s, p}(\Omega)$ and for some constant  $C=C(d,p,s,\Omega)>0$, 
    \begin{equation*}
         \|\xi u\|_{W^{s, p}(\Omega)} \leq C \|u\|_{W^{s, p}(\Omega)}.
    \end{equation*}
\end{lemma}
\begin{proof}
    See  \cite[Lemma 5.3]{di2012hitchhikers} for the proof.
\end{proof}

\smallskip

The next lemma establishes compactness result involving the fractional Sobolev spaces  $W^{s,p}(\Omega)$ in bounded Lipschitz domains. This lemma is useful in proving Theorem  \ref{theorem 3 sp>k}.

\begin{lemma}\label{compactness}
    Let  $s \in (0,1), ~ p \geq 1, ~ \tau \in [1,p], ~ \Omega$ be a bounded Lipschitz domain in $\mathbb{R}^d$, and  $\mathcal{F}$ be a bounded subset of  $L^{p}(\Omega)$. Suppose that
    \begin{equation*}
        \sup_{u \in \mathcal{F}} \int_{\Omega} \int_{\Omega} \frac{|u(x)-u(y)|^{p}}{|x-y|^{d+sp}} \, dx \, dy < + \infty.
    \end{equation*}
    Then  $\mathcal{F}$ is pre-compact in  $L^{\tau}(\Omega)$.
\end{lemma}
\begin{proof}
    See  \cite[Theorem 7.1]{di2012hitchhikers} for the proof.
\end{proof}

\smallskip

The next lemma establishes a basic inequality. Similar inequality is also used in  \cite{squassina2018}  (see  \cite[ Lemma 3.2]{squassina2018}). 

\begin{lemma}\label{estimate}
    Let  $\tau >1$ and $c>1$. Then for all  $a, ~ b \in \mathbb{R}$, we have 
    \begin{equation}
        (|a| + |b|)^{\tau} \leq c|a|^{\tau} + (1-c^{\frac{-1}{\tau -1}})^{1-\tau} |b|^{\tau} .
    \end{equation}
\end{lemma}
\begin{proof}
    Let  $f(x) = (1+x)^{\tau}- cx^{\tau}$ for all  $x \geq 0$. Then  $f(0) =  1$ and  $f(x) \to - \infty$ as  $x \to \infty$. Let  $x_{0}$ be the point of maxima. Then  $f'(x_{0}) = 0$, i.e.,
    \begin{equation*}
        \tau (1+x_{0})^{\tau -1} - c \tau x^{\tau -1}_{0} = 0.
    \end{equation*}
    Therefore, we have  $(1+x_{0})^{\tau}= c^{ \frac{\tau}{\tau-1}} x^{\tau}_{0}$ and  $x_{0} = \frac{1}{c^{\frac{1}{\tau-1}}-1}$. As  $x_{0}$ is the point of maxima, we have  $ f(x) \leq f(x_{0})$ for all  $x \geq 0$. Therefore,
    \begin{equation*}
        (1+x)^{\tau} - cx^{\tau} \leq (1+x_{0})^{\tau} - cx^{\tau}_{0} 
        = (c^{ \frac{\tau}{\tau -1}}- c) x^{\tau}_{0} 
        = \frac{(c^{\frac{\tau}{\tau -1}} - c)} {(c^{ \frac{1}{\tau-1}} -1)^{\tau}} 
        = (1-c^{\frac{-1}{\tau -1}})^{1-\tau}.
    \end{equation*}
    Assume  $|b| \neq 0$ and take  $x= \frac{|a|}{|b|}$. Then, the above inequality proves the lemma.
\end{proof}


\section{The flat boundary case}\label{The flat boundary case}
In this section, we establish our main results for the flat boundary of codimension  $k$, where $1<k<d$ with $k \in \mathbb{N}$ case when the domain is  $B^{k}_{1}(0) \times (-n,n)^{d-k}$, where  $B^{k}_{1}(0)$ is a ball of radius  $1$ with center  $0$ in  $\mathbb{R}^{k}, ~  n \in \mathbb{N}$ with  $K= \{ (x_{k}, x_{d-k}) \in \mathbb{R}^{d} : x_{k} = 0 \ \text{and} \ x_{d-k} \in [-n,n]^{d-k} \}$. In fact, for the case  $k=d-1$, we establish fractional Hardy inequality with singularity on  $K$ on a infinite cylinder (say  $\Omega = B^{d-1}_{1}(0) \times (-\infty,\infty)$ with  $K = \{ x= (x_{d-1}, x_{1}) \in \Omega : x_{d-1}=0  \}$) as the constant obtained for a flat boundary case does not depend on  $n$. For the case  $k=d$, we establish appropriate fractional Hardy inequality with singularity at a point (origin) on a ball of radius  $R>0$ with centre  $0$. First we obtain the fractional Hardy inequality with singularity on  $K$ for the case  $sp=k$, and then the case  $sp<k$ and  $sp>k$ considered seperately. To prove our main results we will do the patching business which is done in the next section.

\smallskip

For  $1 < k < d$ with $k \in \mathbb{N}$, let  $x = (x_{k}, x_{d-k}) \in \mathbb{R}^{d}$, where  $x_{k} \in \mathbb{R}^k$ and  $x_{d-k} \in \mathbb{R}^{d-k}$.  Let  $\alpha = d+ \left(sp-d \right)\frac{ \tau }{p}$, where we fix  $\tau \geq p$ when  $ sp=d$, and  $\tau \in [p, p^{*}_{s}] $ when  $sp<d$. Define, 
\begin{equation*}
    \Omega_{n} := \left\{ x= (x_{k}, x_{d-k}) :  |x_{k}| < 1 \ \text{and} \ x_{d-k} \in (-n, n)^{d-k}   \right\},
\end{equation*}
for some  $n \in \mathbb{N}$. For each  $\ell \leq - 1$, we define (see Figure \ref{fig:myfigure2})
\begin{equation*}
    A_{\ell} = \{ x=(x_{k},x_{d-k}) \in \Omega_{n} :  2^{ \ell} \leq |x_{k}| < 2^{\ell +1} \}.
\end{equation*}
\begin{figure}[!ht]
\begin{center}
\begin{tikzpicture}
\node at (-5,.5) {$2^{\ell} \leq |x_{k}| < 2^{\ell +1}$};
\node at (1.5,1.5) {$(-n,n)^{d-k}$};

\draw[dashed] (-3,1) -- (3,1);
\draw[dashed] (-3,-1) -- (3,-1);

\draw (-3,0) ellipse [y radius=1, x radius=0.5];
\draw (3,0) ellipse [y radius=1, x radius=0.5];
\fill[blue!20, opacity=0.5] (-3,0) ellipse [y radius=1, x radius=0.5];
\fill[blue!20, opacity=0.5] (3,0) ellipse [y radius=1, x radius=0.5];

\fill[white, opacity=1] (-3,0) ellipse [y radius=0.5, x radius=0.25];
\fill[white, opacity=1] (3,0) ellipse [y radius=0.5, x radius=0.25];

\draw[opacity=0.5] (-3,0) ellipse [y radius=.5, x radius=0.25];
\draw[opacity=0.5] (3,0) ellipse [y radius=.5, x radius=0.25];

\foreach \x in {-3,-2.8,...,3}
\draw[blue!20,opacity=0.5] (\x,0) ellipse [y radius=1, x radius=0.5];
    
\foreach \x in {-3,-2.8,...,3}
\draw[blue!20,opacity=0.5] (\x,0) ellipse [y radius=.5, x radius=0.25];
    
\draw[->] (-4,0) -- (4,0) node[right] {$\mathbb{R}^{d-k}$};
\draw[->] (0,-1.75) -- (0,1.75) node[above] {$\mathbb{R}^{k}$};
\end{tikzpicture}
\end{center}
\caption{$A_{\ell}$.}
\label{fig:myfigure2}
\end{figure}

\noindent Therefore, we have
\begin{equation*}
    \Omega_{n} = \bigcup_{\ell= - \infty}^{-1} A_{\ell}.
\end{equation*}
Again, we further divide  $A_{\ell}$ into a disjoint union of identical sets, denoted as  $A^{i}_{\ell}$ (see Figure \ref{fig:myfigure3}), such that if  $x = (x_{k}, x_{d-k}) \in A^{i}_{\ell}$, then  $x_{d-k} \in C^{i}_{\ell}$, where  $C^{i}_{\ell}$  is a cube of side length  $2^{ \ell}$ in $\mathbb{R}^{d-k}$. Then, we have
\begin{equation*}
      A_{\ell} = \bigcup_{i = 1}^{\sigma_{\ell}} A^{i}_{\ell}  ,
\end{equation*}
where  $\sigma_{\ell} := 2^{(- \ell+1)(d-k)} n^{d-k}$. The Lebesgue measure of  $A^{i}_{\ell}$ is then given by
\begin{equation*}
    |A^{i}_{\ell}| = \left( \frac{|\mathbb{S}^{k-1}|(2^{k}-1)2^{\ell k}}{k} \right) \times 2^{\ell (d-k)} =  \frac{|\mathbb{S}^{k-1}|(2^{k}-1)2^{\ell d}}{k}  .
\end{equation*}
\begin{figure}[!ht] 
\begin{center}
\begin{tikzpicture}
\node at (-2,.5) {$2^{\ell} \leq |x_{k}| < 2^{\ell +1}$};
    
\draw (0,1) -- (3,1);
\draw (0,-1) -- (3,-1);
\draw[dashed] (0,.5) -- (3,.5);
\draw[dashed] (0,-.5) -- (3,-.5);

\draw (0,0) ellipse [y radius=1, x radius=0.5];
\fill[blue!20, opacity=0.5] (0,0) ellipse [y radius=1, x radius=0.5];
\fill[white, opacity=1] (0,0) ellipse [y radius=0.5, x radius=0.25];
\draw (0,0) ellipse [y radius=.5, x radius=0.25];

\draw (3,0) ellipse [y radius=1, x radius=0.5];
\fill[blue!20, opacity=0.5] (3,0) ellipse [y radius=1, x radius=0.5];
\fill[white, opacity=1] (3,0) ellipse [y radius=0.5, x radius=0.25];
\draw (3,0) ellipse [y radius=.5, x radius=0.25];
\node at (1.5,1.5) {$C^{i}_{\ell}$};
\end{tikzpicture}
\end{center}
\caption{$A^{i}_{\ell}= \{ x=(x_{k},x_{d-k}) \in A_{\ell} : x_{d-k} \in C^{i}_{\ell} \}   $.}
\label{fig:myfigure3}
\end{figure}

Definition  \ref{definition}  is stated for the domain  $Q= B^{k}_{1}(0) \times (0,1)^{d-k}$. However, we are considering the domain  $\Omega_{n}$ with  $K= \{ (x_{k}, x_{d-k}) \in \mathbb{R}^{d} : x_{k} = 0 \ \text{and} \ x_{d-k} \in [-n,n]^{d-k} \}$ and proving our main results, finding that the constant does not depend on  $n$. Therefore, we are proving our main results for a large domain with a flat boundary of codimension  $k$.  In fact, as  $n \to \infty$, we prove our main results for a flat boundary of codimension  $k, ~ 1<k<d, ~ k \in \mathbb{N}$ in a unbounded domain. One can also assume the domain to be $Q$ with  $K= \{ (x_{k}, x_{d-k}) \in \mathbb{R}^{d} : x_{k} = 0 \ \text{and} \ x_{d-k} \in [0,1]^{d-k} \}$ and still establish the main results using similar steps as done for  $\Omega_{n}$.

\smallskip

Using the above decomposition of $\Omega_{n}$ into the sets $A^{i}_{\ell}$, the proof becomes straightforward. We work on each $A^{i}_{\ell}$ separately and then combine all these estimates through a summation argument to obtain the fractional Hardy inequality on $\Omega_n$ with singularity on $K$. This decomposition is an important part of our  method. For other types of decompositions in the nonlocal case, we refer to the works of Nguyen and Squassina \cite{squassina2018}, and for the local case, in the context of Poincar\'e type inequalities, to the works of Lu \cite{Lu1992, Lu1994}.

\smallskip

The next lemma establishes an inequality that holds true for each  $A^{i}_{\ell}$. This lemma is helpful in establishing our main results when the domain is  $\Omega_{n}$ and  $K= \{ (x_{k}, x_{d-k}) \in \mathbb{R}^{d} : x_{k} = 0 \ \text{and} \ x_{d-k} \in [-n,n]^{d-k} \}$.

\begin{lemma}\label{sumineqlemma}
For any  $A^{i}_{\ell}$, there exists a constant  $C=C(d,p,s, \tau,k)>0$ such that
   \begin{equation}
 \int_{A^{i}_{\ell}} \frac{|u(x)|^{\tau}}{|x_{k}|^{\alpha}} \, dx \leq C [u]^{\tau}_{W^{s, p}(A^{i}_{\ell})} + C 2^{\ell (d-\alpha)} |(u)_{A^{i}_{\ell}}|^{\tau}  .
\end{equation} 
\end{lemma}
\begin{proof}
    Fix any   $A^{i}_{\ell}$. Applying Lemma  \ref{sobolev} with  $\Omega = \{ (x_{k}, x_{d-k}) : 1 < |x_{k}| < 2  \ \text{and} \ x_{d-k} \in (1,2)^{d-k} \}$, and  $ \lambda = 2^{\ell}$ and using translation invariance, we have
\begin{equation*}
    \fint_{A^{i}_{\ell}} |u(x)-(u)_{A^{i}_{\ell}}|^{\tau} \,  dx \leq C 2^{\ell \left(sp-d \right) \frac{ \tau}{p}} [u]^{\tau}_{W^{s, p}(A^{i}_{\ell})}  ,
\end{equation*}
where  $C= C(d,p,s, \tau)$ is a positive constant. For  $x = (x_{k},x_{d-k}) \in A^{i}_{\ell}$, we have  $ \frac{1}{|x_{k}|} \leq \frac{1}{2^{\ell}}$. Therefore, we have
\begin{align*}
    \int_{A^{i}_{\ell}} \frac{|u(x)|^{\tau}}{|x_{k}|^{\alpha}} \,  dx & \leq \frac{1}{2^{\ell \alpha}} \int_{A^{i}_{\ell}} |u(x)-(u)_{A^{i}_{\ell}} + (u)_{A^{i}_{\ell}}|^{\tau} \, dx \\ &
    \leq \frac{2^{\tau-1}}{2^{\ell \alpha}} \int_{A^{i}_{\ell}} |u(x)-(u)_{A^{i}_{\ell}}|^{\tau} \,  dx + \frac{2^{\tau-1}}{2^{\ell \alpha}} \int_{A^{i}_{\ell}} |(u)_{A^{i}_{\ell}}|^{\tau} \, dx.
    \end{align*}
  Now, using the previous inequality, together with the definition of  $\alpha$, we obtain
    \begin{align*}
    \int_{A^{i}_{\ell}} \frac{|u(x)|^{\tau}}{|x_{k}|^{\alpha}} \,  dx & \leq C \frac{|A^{i}_{\ell}|}{2^{\ell \alpha}} \fint_{A^{i}_{\ell}} |u(x)-(u)_{A^{i}_{\ell}}|^{\tau} \, dx + C \frac{|A^{i}_{\ell}|}{2^{\ell \alpha}} |(u)_{A^{i}_{\ell}}|^{\tau} \\ &
    \leq C \left( \frac{|\mathbb{S}^{k-1}|(2^{k}-1)2^{\ell d}}{k} \right) \frac{1}{2^{\ell \alpha}} 2^{\ell \left(sp-d \right) \frac{\tau}{p}} [u]^{\tau}_{W^{s, p}(A^{i}_{\ell})} \\ & \quad  +  \left( \frac{|\mathbb{S}^{k-1}|(2^{k}-1)}{k} \right) 2^{\ell(d-\alpha)} |(u)_{A^{i}_{\ell}}|^{\tau}  
   \\ & \leq C [u]^{\tau}_{W^{s, p}(A^{i}_{\ell})} + C 2^{\ell(d-\alpha)} |(u)_{A^{i}_{\ell}}|^{\tau}  ,
\end{align*}   
where  $C= C(d,p,s,\tau,k)$ is a positive constant. This proves the lemma.
\end{proof}

\smallskip

The following lemma establishes a connection between the averages of two disjoint sets  $A^{j}_{\ell+1}$ and  $A^{i}_{\ell}$, where  $A^{j}_{\ell+1}$ is chosen in such a way that  $C^{i}_{\ell} \subset C^{j}_{\ell+1}$ (see Figure \ref{fig:myfigure4}), i.e., 
\begin{equation}\label{codn on Al and Al+1}
    A^{i}_{\ell} \subset \{ (x_{k}, x_{d-k}) \in A_{\ell} : x_{d-k} \in C^{j}_{\ell +1}   \}.
\end{equation}
\begin{figure}[!ht] 
\begin{center}
\begin{tikzpicture}

\draw (0,1.5) -- (6,1.5);
\draw (0,-1.5) -- (6,-1.5);

\draw (2,1) -- (4,1);
\draw (2,-1) -- (4,-1);

\draw (0,0) ellipse [y radius=1.5, x radius=1];
\fill[blue!20, opacity=0.5] (0,0) ellipse [y radius=1.5, x radius=1];
\fill[white, opacity=1] (0,0) ellipse [y radius=1, x radius=0.5];
\draw (0,0) ellipse [y radius=.5, x radius=0.25];
\draw (2,0) ellipse [y radius=1, x radius=0.5];
\fill[blue!20, opacity=0.5] (2,0) ellipse [y radius=1, x radius=.5];
\fill[green!20, opacity=0.5] (2,0) ellipse [y radius=1, x radius=.5];
\fill[white, opacity=1] (2,0) ellipse [y radius=.5, x radius=0.25];
\draw (0,0) ellipse [y radius=1, x radius=0.5];
\draw (2,0) ellipse [y radius=.5, x radius=0.25];

\draw (6,0) ellipse [y radius=1.5, x radius=1];    
\fill[blue!20, opacity=0.5] (6,0) ellipse [y radius=1.5, x radius=1];
\fill[white, opacity=1] (6,0) ellipse [y radius=1, x radius=0.5];
\draw (6,0) ellipse [y radius=.5, x radius=0.25];
\draw (4,0) ellipse [y radius=1, x radius=0.5];
\fill[blue!20, opacity=0.5] (4,0) ellipse [y radius=1, x radius=.5];
\fill[green!20, opacity=0.5] (4,0) ellipse [y radius=1, x radius=.5];
\fill[white, opacity=1] (4,0) ellipse [y radius=.5, x radius=0.25];
\draw (6,0) ellipse [y radius=1, x radius=0.5];
\draw (4,0) ellipse [y radius=.5, x radius=0.25];

\draw[dashed] (0,1) -- (6,1);
\draw[dashed] (0,-1) -- (6,-1);
\draw[dashed] (0,-.5) -- (6,-.5); 
\draw[dashed] (0,.5) -- (6,.5);

\node at (2,.7) { $A^{i}_{\ell}$ };
\node at (0,1.22) {$A^{j}_{\ell +1}$};
\node at (2.4,2) {$C^{j}_{\ell+1}$};
\node at (3,1.25) {$C^{i}_{\ell}$};
\end{tikzpicture}
\end{center}
\caption{$A^{i}_{\ell}$ and $A^{j}_{\ell +1}$ .}
\label{fig:myfigure4}
\end{figure}
 
\noindent Also, there are $2^{d-k}$ such $A^{i}_{\ell}$'s which satisfies the above relation. This lemma is helpful in proving the main results when the domain is  $\Omega_{n}$ and  $K= \{ (x_{k}, x_{d-k}) \in \mathbb{R}^{d} : x_{k} = 0 \ \text{and} \ x_{d-k} \in [-n,n]^{d-k} \}$.

\begin{lemma}\label{est2}
    Let  $A^{j}_{\ell+1}$ be a set such that   $A^{i}_{\ell}$ lies below the set  $A^{j}_{\ell+1}$. Then
    \begin{equation}
        |(u)_{A^{i}_{\ell}} - (u)_{A^{j}_{\ell+1}} |^{\tau} \leq C 2^{\ell \left(sp-d \right) \frac{ \tau}{p}} [u]^{\tau}_{W^{s, p}(A^{i}_{\ell} \cup A^{j}_{\ell+1})},
    \end{equation}
    where  $C$ is a positive constant independent of  $\ell$.
\end{lemma}
\begin{proof}
Let us consider  $|(u)_{A^{i}_{\ell}} - (u)_{A^{j}_{\ell+1}} |^{\tau}$. Using Lemma  \ref{avg} with  $E=A^{i}_{\ell}$ and  $F = A^{j}_{\ell+1}$, we obtain, for some constant $C$ (independent of  $\ell$),
\begin{equation}\label{compineq1}
|(u)_{A^{i}_{\ell}} - (u)_{A^{j}_{\ell+1}} |^{\tau} \leq C \fint_{A^{i}_{\ell} \cup A^{j}_{\ell+1} }  |u(x) - (u)_{A^{i}_{\ell} \cup A^{j}_{\ell+1}}|^{\tau} \, dx.
\end{equation}
Choose an open set  $\Omega$ such that  $\Omega_{\lambda}$ is a translation of  $ A^{i}_{\ell} \cup A^{j}_{\ell+1}$ with scaling parameter  $\lambda=2^{\ell+1}$. Applying Lemma  \ref{sobolev} with this  $\Omega$ and  $\lambda=2^{\ell+1}$, and using translation invariance, we obtain
\begin{equation}\label{compineq2}
     \fint_{A^{i}_{\ell} \cup A^{j}_{\ell+1} }  |u(x) - (u)_{A^{i}_{\ell} \cup A^{j}_{\ell+1}}|^{\tau} \, dx \leq C 2^{\ell \left(sp-d \right) \frac{ \tau}{p}} [u]^{\tau}_{W^{s, p}(A^{i}_{\ell} \cup A^{j}_{\ell+1})}.
\end{equation}
By combining  \eqref{compineq1} and  \eqref{compineq2}, we establish the lemma.
\end{proof}

\smallskip

The next lemma proves Theorem  \ref{theorem 1 sp=k} for the domain  $\Omega_{n}$ and  $K= \{ (x_{k}, x_{d-k}) \in \mathbb{R}^{d} : x_{k} = 0 \ \text{and} \ x_{d-k} \in [-n,n]^{d-k} \}$ for  $\tau \in \left[p, \frac{dp+1}{1-(k-d)} \right]$. For general bounded Lipschitz domain  $\Omega$ with smooth compact set  $K \subset \Omega $ of codimension  $k$, where $1<k<d, ~ k \in \mathbb{N}$, the proof follows by the usual patching technique using partition of unity.

\begin{lemma}\label{flat case sp=k 1}
     Let   $1<k<d$ with $k \in \mathbb{N}$ and   $sp=k$. Then for any   $\tau \in [p,p^{*}_{s} ] $ and  $x= (x_{k}, x_{d-k}) \in \Omega_{n}$, with $x_{k} \in \mathbb{R}^{k}$, there exists a constant  $C= C(d,p, \tau,k) > 0$  such that
    \begin{equation}
     \left(  \bigintsss_{\Omega_{n}} \frac{|u(x)|^{\tau}}{|x_{k}|^{\alpha} \ln^{b} \left(\frac{2}{|x_{k}|} \right)}  dx \right)^{\frac{1}{\tau}} \leq  C 
 \|u\|_{W^{s, p}(\Omega_{n})}, \quad \forall \ u \in W^{s,p}_{0}(\Omega_n \setminus K),
    \end{equation}
     where  $b= \tau-1$ when  $0 \leq \alpha <k$, and  $b= \tau = p$ when  $\alpha =k$.
\end{lemma}
\begin{proof}
Let  $u \in W^{s,p}(\Omega_{n}) \cap C(\Omega_{n})$ be such that  $u=0$ in a neighbourhood of  $K$. For each  $x= (x_{k}, x_{d-k}) \in A^{i}_{\ell}$, we have  $|x_{k}| < 2^{\ell +1}$, which  implies  $ \ln \left(\frac{2}{|x_{k}|} \right) > (-\ell) \ln 2$. Using this together with Lemma  \ref{sumineqlemma} and the fact that  $\frac{1}{(-\ell)^{b}} \leq 1$, we obtain
\begin{align*}
    \bigintsss_{A^{i}_{\ell}} \frac{|u(x)|^{\tau}}{|x_{k}|^{\alpha} \ln^{b} \left(\frac{2}{|x_{k}|} \right)} \, dx & \leq \frac{C}{(-\ell)^{b}} [u]^{\tau}_{W^{s, p}(A^{i}_{\ell})} + C \frac{2^{\ell(d-\alpha)}}{(-\ell)^{b}} |(u)_{A^{i}_{\ell}}|^{\tau} \\ &
     \leq  C [u]^{\tau}_{W^{s, p}(A^{i}_{\ell})} + C \frac{2^{\ell(d-\alpha)}}{(-\ell)^{b}} |(u)_{A^{i}_{\ell}}|^{\tau}  .
\end{align*}
By summing the above inequality from  $i=1$ to  $\sigma_{\ell}$, we obtain
\begin{equation*}
    \bigintsss_{A_{\ell}} \frac{|u(x)|^{\tau}}{|x_{k}|^{\alpha} \ln^{b} \left( \frac{2}{|x_{k}|} \right)} \,  dx \leq C   \sum_{i=1}^{\sigma_{\ell}} [u]^{\tau}_{W^{s, p}(A^{i}_{\ell})} + C \frac{2^{\ell (d-\alpha)}}{(-\ell)^{b}}  \sum_{i=1}^{\sigma_{\ell}} |(u)_{A^{i}_{\ell}}|^{\tau}.
\end{equation*}
Applying  \eqref{sumineq} with  $\gamma= \frac{\tau}{p}$, we have
\begin{equation}\label{seminorm ineq relation}
    \sum_{i=1}^{\sigma_{\ell}} [u]^{\tau}_{W^{s, p}(A^{i}_{\ell})} \leq \left( \sum_{i=1}^{\sigma_{\ell}} [u]^{p}_{W^{s, p}(A^{i}_{\ell})} \right)^{\frac{\tau}{p}} \leq [u]^{\tau}_{W^{s, p}(A_{\ell})} .
\end{equation}
Therefore, we have
\begin{equation*}
    \bigintsss_{A_{\ell}} \frac{|u(x)|^{\tau}}{|x_{k}|^{\alpha} \ln^{b} \left( \frac{2}{|x_{k}|} \right) } \, dx \leq C [u]^{\tau}_{W^{s, p}(A_{\ell})} +  C \frac{2^{\ell (d-\alpha)}}{(-\ell)^{b}}  \sum_{i=1}^{\sigma_{\ell}} |(u)_{A^{i}_{\ell}}|^{\tau}.
\end{equation*}
Again, summing the above inequality from  $\ell =m \in \mathbb{Z}^{-}$ to  $-1$, we obtain
\begin{equation}\label{eqnn2}
\sum_{\ell=m}^{-1} \bigintsss_{A_{\ell}} \frac{|u(x)|^{\tau}}{|x_{k}|^{\alpha} \ln^{b}\left( \frac{2}{|x_{k}|} \right)} \,  dx \leq  C \sum_{\ell=m}^{-1} [u]^{\tau}_{W^{s, p}(A_{\ell})} + C \sum_{\ell=m}^{-1} \frac{2^{\ell (d-\alpha)}}{(-\ell)^{b}}  \sum_{i=1}^{\sigma_{\ell}} |(u)_{A^{i}_{\ell}}|^{\tau}  .
\end{equation}
Let  $A^{j}_{\ell+1}$ be such that it satisfies  \eqref{codn on Al and Al+1}. Using triangle inequality, we have 
\begin{equation*}
    |(u)_{A^{i}_{\ell}}|^\tau \leq  \left( |(u)_{A^{j}_{\ell+1}}| + |(u)_{A^{i}_{\ell}} - (u)_{A^{j}_{\ell+1}}| \right)^\tau.
\end{equation*}
For  $\ell \in \mathbb{Z}^{-} \setminus \{-1 \}$, applying Lemma  \ref{estimate} with  $c := \left( \frac{-\ell}{-\ell-(1/2)} \right)^{\tau-1}>1$ together with Lemma  \ref{est2} and  $sp=k$, we obtain
\begin{equation*}
    |(u)_{A^{i}_{\ell}}|^{\tau} \leq \left( \frac{-\ell}{-\ell-(1/2)} \right)^{\tau-1} |(u)_{A^{j}_{\ell+1}}|^{\tau} + C (-\ell)^{\tau-1} 2^{\ell(k-d) \frac{ \tau}{p}} [u]^{\tau}_{W^{s, p}(A^{i}_{\ell} \cup A^{j}_{\ell+1})}  .
\end{equation*}
Multiplying the above inequality by  $2^{\ell (d-\alpha)}$, and using the definition of  $\alpha$, we get
\begin{equation*}
    \frac{2^{\ell (d-\alpha)}}{(-\ell)^{\tau-1}}  |(u)_{A^{i}_{\ell}}|^{\tau} \leq \frac{2^{\ell (d-\alpha)}}{ (-\ell-(1/2))^{\tau-1}} |(u)_{A^{j}_{\ell+1}}|^{\tau} + C [u]^{\tau}_{W^{s, p}(A^{i}_{\ell} \cup A^{j}_{\ell+1})}  .
\end{equation*}
Since there are  $2^{d-k}$ such  $A^{i}_{\ell}$'s that satisfy \eqref{codn on Al and Al+1}, summing the above inequality from  $i=2^{d-k}(j-1)+1$ to  $2^{d-k}j$, we obtain
\begin{align}\label{cubes below the cubes 1}
      \frac{2^{\ell (d-\alpha)}}{(-\ell)^{\tau-1}}  \sum_{i=2^{d-k}(j-1)+1}^{2^{d-k}j} |(u)_{A^{i}_{\ell}}|^{\tau} & \leq 2^{d-k} \frac{2^{\ell (d-\alpha)}}{ (-\ell-(1/2))^{\tau-1}} |(u)_{A^{j}_{\ell+1}}|^{\tau} \nonumber \\ & \quad
       + \ C  \sum_{i=2^{d-k}(j-1)+1}^{2^{d-k}j} [u]^{\tau}_{W^{s, p}(A^{i}_{\ell} \cup A^{j}_{\ell+1})}  .
\end{align}
\underline{\textbf{Case 1: \  $\alpha=k$}:} \ In this case, we have  $\tau =p$. From \eqref{cubes below the cubes 1}, we have

\begin{align*}
      \frac{2^{\ell(d-k)}}{(-\ell)^{p-1}}  \sum_{i=2^{d-k}(j-1)+1}^{2^{d-k}j} |(u)_{A^{i}_{\ell}}|^{p} & \leq 2^{d-k} \frac{2^{\ell(d-k)}}{ (-\ell-(1/2))^{p-1}} |(u)_{A^{j}_{\ell+1}}|^{p} \\ & \quad
      + \ C  \sum_{i=2^{d-k}(j-1)+1}^{2^{d-k}j} [u]^{p}_{W^{s, p}(A^{i}_{\ell} \cup A^{j}_{\ell+1})}  .
\end{align*}
Again, summing the above inequality from  $j=1$ to  $\sigma_{\ell+1}$, and using the fact that
\begin{equation}\label{summation relation}
  \sum_{j=1}^{\sigma_{\ell+1}} \Bigg(  \sum_{i=2^{d-k}(j-1)+1}^{2^{d-k}j}  |(u)_{A^{i}_{\ell}}|^{p} \Bigg)  = \sum_{i=1}^{\sigma_{\ell}}  |(u)_{A^{i}_{\ell}}|^{p},
\end{equation}
we obtain
\begin{align*}
\frac{2^{\ell(d-k)}}{(-\ell)^{p-1}}  \sum_{i=1}^{\sigma_{\ell}} |(u)_{A^{i}_{\ell}}|^{p} & \leq \frac{2^{(\ell+1)(d-k)}}{ (-\ell-(1/2))^{p-1}} \sum_{j=1}^{\sigma_{\ell+1}} |(u)_{A^{j}_{\ell+1}}|^{p} \\ & \quad
  + \ C \sum_{j=1}^{\sigma_{\ell+1}} \Bigg( \sum_{i=2^{d-k}(j-1)+1}^{2^{d-k}j}  [u]^{p}_{W^{s, p}(A^{i}_{\ell} \cup A^{j}_{\ell+1})} \Bigg).
\end{align*}
Also, we have
\begin{equation*}
    \sum_{j=1}^{\sigma_{\ell+1}} \Bigg( \sum_{i=2^{d-k}(j-1)+1}^{2^{d-k}j}  [u]^{p}_{W^{s, p}(A^{i}_{\ell} \cup A^{j}_{\ell+1})} \Bigg)   \leq [u]^{p}_{W^{s, p}(A_{\ell} \cup A_{\ell+1})} .
\end{equation*}
Combining the above two inequalities, we obtain
\begin{equation*}
   \frac{2^{\ell(d-k)}}{(-\ell)^{p-1}}  \sum_{i=1}^{\sigma_{\ell}} |(u)_{A^{i}_{\ell}}|^{p}  \leq \frac{2^{(\ell+1)(d-k)}}{ (-\ell-(1/2))^{p-1}} \sum_{j=1}^{\sigma_{\ell+1}} |(u)_{A^{j}_{\ell+1}}|^{p} + C  [u]^{p}_{W^{s, p}(A_{\ell} \cup A_{\ell+1})}  .
\end{equation*}
For simplicity, let  $a_{\ell} = \sum_{i=1}^{\sigma_{\ell}} |(u)_{A^{i}_{\ell}}|^{p}$. Then, the above inequality will become
\begin{equation*}
    \frac{2^{\ell(d-k)}}{(-\ell)^{p-1}} a_{\ell} \leq \frac{2^{(\ell+1)(d-k)}}{ (-\ell-(1/2))^{p-1}} a_{\ell+1} + C  [u]^{p}_{W^{s, p}(A_{\ell} \cup A_{\ell+1})}  .
\end{equation*}
Summing the above inequality from  $\ell=m \in \mathbb{Z}^{-}$ to  $-2$, we get
\begin{equation*}
   \sum_{\ell=m}^{-2} \frac{2^{\ell(d-k)}}{(-\ell)^{p-1}} a_{\ell} \leq \sum_{\ell=m}^{-2} \frac{2^{(\ell+1)(d-k)}}{ (-\ell-(1/2))^{p-1}} a_{\ell+1} + C  \sum_{\ell=m}^{-2} [u]^{p}_{W^{s, p}(A_{\ell} \cup A_{\ell+1})}  .
\end{equation*}
By changing sides, rearranging, and re-indexing, we get
\begin{multline*}
    \frac{2^{m(d-k)}}{(-m)^{p-1}} a_{m} + \sum_{\ell=m+1}^{-2} \left\{ \frac{1}{(-\ell)^{p-1}} - \frac{1}{(-\ell+1/2)^{p-1}} \right\} 2^{\ell(d-k)} a_{\ell}  \\ \leq \left( \frac{2}{3} \right)^{p-1} 2^{(-1)(d-k)}a_{-1} 
    + \ C \sum_{\ell=m}^{-2} [u]^{p}_{W^{s, p}(A_{\ell} \cup A_{\ell+1})}  .
\end{multline*}
Now, for large values of  $-\ell$, we have
\begin{equation*}
     \frac{1}{(-\ell)^{p-1}} - \frac{1}{(-\ell+1/2)^{p-1}} \sim \frac{1}{(-\ell)^{p}},
\end{equation*}
and by choosing  $-m$ large enough such that  $|(u)_{A^{j}_{m}}|=0$ for all  $j \in \{ 1, \dots, \sigma_{m} \}$, we obtain
\begin{equation*}
    \sum_{\ell=m}^{-2} \frac{2^{\ell(d-k)}}{(-\ell)^{p}} a_{\ell} \leq  C a_{-1} +  C \sum_{\ell=m}^{-2} [u]^{p}_{W^{s, p}(A_{\ell} \cup A_{\ell+1})}  .
\end{equation*}
Therefore, adding   $2^{(k -d)}a_{-1}$ on  both sides,  we have for some constant  $C= C(d,p, k)>0$ 
\begin{equation*}
    \sum_{\ell=m}^{-1} \frac{2^{\ell(d-k)}}{(-\ell)^{p}} a_{\ell} \leq C  a_{-1} +  C \sum_{\ell=m}^{-2} [u]^{p}_{W^{s, p}(A_{\ell} \cup A_{\ell+1})}  .
\end{equation*}
Putting the value of  $a_{\ell}$ in the above inequality, we obtain
\begin{equation}\label{eqn10}
    \sum_{\ell=m}^{-1} \frac{2^{\ell(d-k)}}{(-\ell)^{p}} \sum_{i=1}^{\sigma_{\ell}} |(u)_{A^{i}_{\ell}}|^{p} \leq C \sum_{j=1}^{\sigma_{-1}} |(u)_{A^{j}_{-1}}|^{p} + C \sum_{\ell=m}^{-2} [u]^{p}_{W^{s, p}(A_{\ell} \cup A_{\ell+1})}  .
\end{equation}
Combining  \eqref{eqnn2} and  \eqref{eqn10} with  $\tau = b= p$ and  $\alpha=k$ yields
\begin{equation*}\label{eqqnn}
    \sum_{\ell=m}^{-1} \bigintsss_{A_{\ell}} \frac{|u(x)|^{p}}{|x_{k}|^{k} \ln^{p} \left( \frac{2}{|x_{k}|} \right)} \, dx \leq  C \sum_{j=1}^{\sigma_{-1}} |(u)_{A^{j}_{-1}}|^{p} +  C \sum_{\ell=m}^{-2} [u]^{p}_{W^{s, p}(A_{\ell} \cup A_{\ell+1})}  .
\end{equation*}
Also,
\begin{equation*}
      \sum_{j=1}^{\sigma_{-1}} |(u)_{A^{j}_{-1}}|^{p} =   \sum_{j=1}^{\sigma_{-1}} \Bigg| \frac{1}{|A^{j}_{-1}|} \int_{A^{j}_{-1}} u(x) \, dx \Bigg|^{p}
     \leq  C \sum_{j=1}^{\sigma_{-1}} \int_{A^{j}_{-1}} |u(x)|^{p} \, dx \leq C \int_{\Omega_{n}} |u(x)|^{p} \, dx.
\end{equation*}
Here, we have used H$\ddot{\text{o}}$lder's inequality with  $ \frac{1}{p} +  \frac{1}{p'} = 1$. Combining the above two inequalities, we obtain
\begin{equation*}
     \sum_{\ell=m}^{-1} \bigintsss_{A_{\ell}} \frac{|u(x)|^{p}}{|x_{k}|^{k} \ln^{p} \left( \frac{2}{|x_{k}|} \right)} \, dx \leq  C \left( \|u\|^{p}_{L^{p}(\Omega_{n})} +  \sum_{\ell=m}^{-2} [u]^{p}_{W^{s, p}(A_{\ell} \cup A_{\ell+1})} \right)  .
\end{equation*}
Therefore, from  \eqref{se1} (see Appendix  \ref{appendix}), we have
 \begin{equation*}
      \left( \bigintsss_{\Omega_{n}} \frac{|u(x)|^{p}}{|x_{k}|^{k} \ln^{p}\left( \frac{2}{|x_{k}|} \right)}  \, dx \right)^{\frac{1}{p}} \leq C \left( [u]^{p}_{W^{s, p}(\Omega_{n})} + \|u\|^{p}_{L^{p}(\Omega_{n})} \right)^{\frac{1}{p}}   .
 \end{equation*}
\underline{\textbf{Case 2: \  $ 0 \leq \alpha<k$}:} \ In this case, we have  $\tau \in ( p, p^{*}_{s}]$. One can modify the proof of the case  $\alpha=k$ to obtain the result for the case  $0 \leq \alpha <k$. Consider the inequality  \eqref{cubes below the cubes 1} and rewrite it
\begin{align*}
      \frac{2^{\ell (d-\alpha)}}{(-\ell)^{\tau-1}}  \sum_{i=2^{d-k}(j-1)+1}^{2^{d-k}j} |(u)_{A^{i}_{\ell}}|^{\tau} & \leq  \frac{2^{(\ell+1)(d-\alpha)}}{ 2^{k- \alpha} (-\ell-(1/2))^{\tau-1}} |(u)_{A^{j}_{\ell+1}}|^{\tau} \\ & \quad
       + \ C  \sum_{i=2^{d-k}(j-1)+1}^{2^{d-k}j} [u]^{\tau}_{W^{s, p}(A^{i}_{\ell} \cup A^{j}_{\ell+1})}  .
\end{align*}
Using the fact
\begin{equation*}
     \frac{1}{(-\ell)^{\tau-1}} - \frac{1}{2^{k-\alpha}(-\ell+1/2)^{\tau-1}} \sim \frac{1}{(-\ell)^{\tau-1}},
\end{equation*}
and following the similar approach illustrated in the case $1$, one can obtain the following inequality:
\begin{equation*}
     \left(  \bigintsss_{\Omega_{n}} \frac{|u(x)|^{\tau}}{|x_{k}|^{\alpha} \ln^{\tau -1} \left(\frac{2}{|x_{k}|} \right)} \, dx \right)^{\frac{1}{\tau}} \leq  C 
 \|u\|_{W^{s, p}(\Omega_{n})},
\end{equation*}
for the case  $0 \leq \alpha <k$.
\end{proof}

The next lemma will improve the weight function presented in Lemma  \ref{flat case sp=k 1}. We improve the weight function for  $\tau \in \left( \frac{dp+1}{1-(k-d)}, p^{*}_{s}\right]$.
\begin{lemma}\label{flat case sp=k 2}
     Let  $1<k<d$ with $k \in \mathbb{N}$ and  $sp=k$. Then for any  $x = (x_{k},x_{d-k}) \in \Omega_{n}$, with $x_{k} \in \mathbb{R}^{k}$, there exist a constant  $C= C(d,p, \tau,k)>0$ such that for any  $\tau \in  \left[ p, p^{*}_{s} \right] $, we have
    \begin{equation}
     \left(  \bigintsss_{\Omega_{n}} \frac{|u(x)|^{\tau}}{|x_{k}|^{\alpha} \ln^{\beta} \left(\frac{2}{|x_{k}|} \right)} \,  dx \right)^{\frac{1}{\tau}} \leq   C \|u\|_{W^{s, p}(\Omega_{n})}, \hspace{3mm} \ \forall \  u \in W^{s,p}_{0}(\Omega_n \setminus K),
    \end{equation}
     where  $\beta= p$ when  $\tau=p,  ~ \beta = \tau -1$ when  $\tau \in \left(p, \frac{dp+1}{1-(k-d)} \right]$, and  $\beta= dp +  (k-d) \tau$ when  $\tau \in \left( \frac{dp+1}{1-(k-d)}, p^{*}_{s}\right]$.
\end{lemma}
\begin{proof}
It is sufficient to prove for  $\tau \in \left( \frac{dp+1}{1-(k-d)}, p^{*}_{s} \right]$. For  $\tau \in  \left[ p, \frac{dp+1}{1-(k-d)} \right]$, it follows from Lemma  \ref{flat case sp=k 1}. Let  $\tau_{1} = \frac{dp+1}{1-(k-d)}$, using Lemma  \ref{flat case sp=k 1}, we have
\begin{equation*}
        \left(  \bigintsss_{\Omega_{n}} \frac{|u(x)|^{\tau_{1}}}{|x_{k}|^{\alpha_{\tau_{1}}} \ln^{\tau_{1} -1} \left(\frac{2}{|x_{k}|} \right)} \,  dx \right)^{\frac{1}{\tau_{1}}} \leq  C \|u\|_{W^{s, p}(\Omega_{n})} ,
    \end{equation*}
where  $\alpha_{\tau_{1}} = d+ (k-d) \frac{\tau_{1}}{p} = \frac{dp+ (k-d)}{p(1-(k-d))}$. Also, from  \eqref{Sobineqq}, we have
    \begin{equation*}
        \left(  \int_{\Omega_{n}} |u(x)|^{p^{*}_{s}} \,  dx \right)^{\frac{1}{p^{*}_{s}}} \leq  C \|u\|_{W^{s, p}(\Omega_{n})} .
    \end{equation*}
    For any  $\tau \in (\tau_{1}, p^{*}_{s}]$, there exists  $\theta \in [0,1)$ such that  $\tau = \theta \tau_{1} + (1- \theta)p^{*}_{s}$. Let  $a_{1} ,b_{1} \geq 0$, then by using H$\ddot{\text{o}}$lder's inequality with  $\frac{1}{1/ \theta}+ \frac{1}{1/(1- \theta)}= 1$, we obtain
    \begin{align*}
        \bigintsss_{\Omega_{n}} \frac{|u(x)|^{\tau}}{ |x_{k}|^{a_{1}} \ln^{b_{1}}\left( \frac{2}{|x_{k}|} \right)} \, dx & =  \bigintsss_{\Omega_{n}} \frac{|u(x)|^{\theta \tau_{1} + (1- \theta)p^{*}_{s}}}{ |x_{k}|^{a_{1}} \ln^{b_{1}}\left( \frac{2}{|x_{k}|} \right) } \, dx \\ & \leq  \Bigg( \bigintsss_{\Omega_{n}} \frac{|u(x)|^{\tau_{1}}}{ |x_{k}|^{\frac{a_{1}}{\theta}} \ln^{\frac{b_{1}}{\theta}}\left( \frac{2}{|x_{k}|} \right)} \,  dx \Bigg)^{\theta}  \Bigg( \int_{\Omega_{n}} |u(x)|^{p^{*}_{s}} \, dx \Bigg)^{1- \theta} .
    \end{align*}
    Let  $\frac{a_{1}}{\theta} = \alpha_{\tau_{1}}$ and  $\frac{b_{1}}{\theta} = \tau_{1} -1 $, we have
    \begin{align*}
         \bigintsss_{\Omega_{n}} \frac{|u(x)|^{\tau}}{ |x_{k}|^{a_{1}} \ln^{b_{1}}\left( \frac{2}{|x_{k}|} \right)} \, dx  & \leq  \Bigg( \bigintsss_{\Omega_{n}} \frac{|u(x)|^{\tau_{1}}}{ |x_{k}|^{ \alpha_{\tau_{1}} }  \ln^{\tau_{1}-1} \left( \frac{2}{|x_{k}|} \right)} \,  dx \Bigg)^{\theta} \Bigg( \int_{\Omega_{n}} |u(x)|^{p^{*}_{s}} \, dx \Bigg)^{1- \theta} \\ & \leq C \|u\|^{\theta \tau_{1}}_{W^{s, p}(\Omega_{n})}  \|u\|^{(1- \theta) p^{*}_{s}}_{W^{s, p}(\Omega_{n})}  \leq
         C \|u\|^{\tau}_{W^{s, p}(\Omega_{n})}.
    \end{align*}
    From the definition of  $\theta$, we obtain  $a_{1} = \alpha$ and  $b_{1} = \beta$. Therefore, we have
    \begin{equation*}
       \Bigg( \bigintsss_{\Omega_{n}} \frac{|u(x)|^{\tau}}{ |x_{k}|^{\alpha} \ln^{\beta}\left( \frac{2}{|x_{k}|} \right)} \,  dx \Bigg)^{\frac{1}{\tau}} \leq C \|u\|_{W^{s, p}(\Omega_{n})},
    \end{equation*}
    for  $\tau \in \left( \frac{dp+1}{1-(k-d)} , p^{*}_{s} \right]$.
\end{proof}

The next lemma proves Theorem  \ref{theorem 2 sp<k} and establishes the fractional Hardy-type inequality for the case  $sp<k$, where $1<k<d, ~ k \in \mathbb{N}$, when the domain is  $\Omega_{n}$ and  $K= \{ (x_{k}, x_{d-k}) \in \mathbb{R}^{d} : x_{k} = 0 \ \text{and} \ x_{d-k} \in [-n,n]^{d-k} \}$. 

\begin{lemma}\label{flat case sp<k}
     Let  $1<k<d$ with $k \in \mathbb{N},  ~p>1$ and  $ sp<k$. For any  $\tau \in [ p, p^{*}_{s}] $ and  $x= (x_{k},x_{d-k}) \in \Omega_{n}$, with $x_{k} \in \mathbb{R}^{k}$, there exists a constant  $C= C(d,p,s, \tau,k) > 0$  such that
    \begin{equation}
     \left(  \bigintssss_{\Omega_{n}} \frac{|u(x)|^{\tau}}{|x_{k}|^{\alpha} } \, dx \right)^{\frac{1}{\tau}} \leq  C 
 \|u\|_{W^{s, p}(\Omega_{n})}, \quad \forall \ u \in W^{s,p}_{0}(\Omega_n \setminus K).
    \end{equation}
\end{lemma}
\begin{proof}
Let  $u \in W^{s,p}(\Omega_{n}) \cap C(\Omega_{n})$ be such that  $u=0$ in a neighbourhood of $K$. From Lemma  \ref{sumineqlemma}, we have
\begin{equation*}
     \int_{A^{i}_{\ell}} \frac{|u(x)|^{\tau}}{|x_{k}|^{\alpha}} \,  dx \leq C [u]^{\tau}_{W^{s, p}(A^{i}_{\ell})} + C 2^{\ell (d-\alpha)} |(u)_{A^{i}_{\ell}}|^{\tau}  .
\end{equation*}
By summing the above inequality from  $i=1$ to  $\sigma_{\ell}$, and using  \eqref{seminorm ineq relation}, we obtain
\begin{align*}
    \int_{A_{\ell}} \frac{|u(x)|^{\tau}}{|x_{k}|^{\alpha}} \, dx & \leq C   \sum_{i=1}^{\sigma_{\ell}} [u]^{\tau}_{W^{s, p}(A^{i}_{\ell})} + C 2^{\ell (d-\alpha)}  \sum_{i=1}^{\sigma_{\ell}} |(u)_{A^{i}_{\ell}}|^{\tau}  \\ &
    \leq C [u]^{\tau}_{W^{s, p}(A_{\ell})} +  C 2^{\ell (d-\alpha)}  \sum_{i=1}^{\sigma_{\ell}} |(u)_{A^{i}_{\ell}}|^{\tau}.
\end{align*}
Again, summing the above inequality from  $\ell=m \in \mathbb{Z}^{-}$ to  $-1$, we get
\begin{equation}\label{eqnn1}
\sum_{\ell=m}^{-1} \int_{A_{\ell}} \frac{|u(x)|^{\tau}}{|x_{k}|^{\alpha}} \,  dx \leq  C \sum_{\ell=m}^{-1} [u]^{\tau}_{W^{s, p}(A_{\ell})} + C \sum_{\ell=m}^{-1} 2^{\ell (d-\alpha)} \sum_{i=1}^{\sigma_{\ell}} |(u)_{A^{i}_{\ell}}|^{\tau}  .
\end{equation} 
For any  $\tau \in \left[p, p^{*}_{s} \right] $, we have  $\alpha \in [0,sp]$ and we know that  $sp< k$. Let  $A^{j}_{\ell+1}$ be such that it satisfies  \eqref{codn on Al and Al+1}. Using triangle inequality, we have 
\begin{equation*}
    |(u)_{A^{i}_{\ell}}|^\tau \leq  \left( |(u)_{A^{j}_{\ell+1}}| + |(u)_{A^{i}_{\ell}} - (u)_{A^{j}_{\ell+1}}| \right)^\tau.
\end{equation*}
For  $1<k<d$ with $ k \in \mathbb{N}$ and  $sp<k$, applying Lemma  \ref{estimate} with  $c :=c_{1}2^{k - \alpha}>1$, where  $c_{1} = \frac{2}{1+ 2^{k- \alpha}} <1$ together with Lemma  \ref{est2}, we obtain
\begin{equation*}
    |(u)_{A^{i}_{\ell}}|^{\tau} \leq c_{1} 2^{k- \alpha} |(u)_{A^{j}_{\ell+1}}|^{\tau} + C  2^{\ell \left(sp-d \right) \frac{ \tau}{p}} [u]^{\tau}_{W^{s, p}(A^{i}_{\ell} \cup A^{j}_{\ell+1})}  .
\end{equation*}
Multiplying the above inequality by  $2^{\ell (d-\alpha)}$, and using the definition of  $\alpha$, we get
\begin{equation*}
   2^{\ell (d-\alpha)} |(u)_{A^{i}_{\ell}}|^{\tau} \leq c_{1} 2^{\ell(d-\alpha) + (k-\alpha)}|(u)_{A^{j}_{\ell+1}}|^{\tau} + C [u]^{\tau}_{W^{s, p}(A^{i}_{\ell} \cup A^{j}_{\ell+1})}  .
\end{equation*}
Since there are  $2^{d-k}$ such  $A^{i}_{\ell}$'s that satisfy \eqref{codn on Al and Al+1}, summing the above inequality from  $i=2^{d-k}(j-1)+1$ to $2^{d-k}j$, we obtain
\begin{align*}
     2^{\ell (d-\alpha)}  \sum_{i=2^{d-k}(j-1)+1}^{2^{d-k}j} |(u)_{A^{i}_{\ell}}|^{\tau} & \leq c_{1} 2^{d-k} 2^{\ell(d-\alpha) + (k- \alpha)} |(u)_{A^{j}_{\ell+1}}|^{\tau} \\ & \quad
       + \ C  \sum_{i=2^{d-k}(j-1)+1}^{2^{d-k}j} [u]^{\tau}_{W^{s, p}(A^{i}_{\ell} \cup A^{j}_{\ell+1})}  .
\end{align*}
Again, summing the above inequality from  $j=1$ to  $\sigma_{\ell+1}$, and using  \eqref{summation relation}, we obtain
\begin{align}\label{ineqn100}
2^{\ell (d-\alpha)}  \sum_{i=1}^{\sigma_{\ell}} |(u)_{A^{i}_{\ell}}|^{\tau} & \leq c_{1} 2^{(\ell+1)(d-\alpha)} \sum_{j=1}^{\sigma_{\ell+1}} |(u)_{A^{j}_{\ell+1}}|^{\tau} \nonumber \\ &
 \quad + \ C \sum_{j=1}^{\sigma_{\ell+1}} \Bigg( \sum_{i=2^{d-k}(j-1)+1}^{2^{d-k}j}  [u]^{\tau}_{W^{s, p}(A^{i}_{\ell} \cup A^{j}_{\ell+1})} \Bigg).
\end{align}
Applying  \eqref{sumineq} with  $\gamma = \frac{\tau}{p}$ on the last term of the above equation, we have
\begin{align}\label{ineqn22}
    \sum_{j=1}^{\sigma_{\ell+1}} \Bigg( \sum_{i=2^{d-k}(j-1)+1}^{2^{d-k}j}  [u]^{\tau}_{W^{s,p}(A^{i}_{\ell} \cup A^{j}_{\ell+1})} \Bigg) & \leq \Bigg( \sum_{j=1}^{\sigma_{\ell+1}} \Bigg( \sum_{i=2^{d-k}(j-1)+1}^{2^{d-k}j}  [u]^{p}_{W^{s,p}(A^{i}_{\ell} \cup A^{j}_{\ell+1})} \Bigg) \Bigg)^{\frac{\tau}{p}} \nonumber \\ &  \leq [u]^{\tau}_{W^{s,p}(A_{\ell} \cup A_{\ell+1})} .
\end{align}
Therefore, we obtain
\begin{equation*}
  2^{\ell (d-\alpha)} \sum_{i=1}^{\sigma_{\ell}} |(u)_{A^{i}_{\ell}}|^{\tau}  \leq c_{1} 2^{(\ell+1)(d-\alpha)} \sum_{j=1}^{\sigma_{\ell+1}} |(u)_{A^{j}_{\ell+1}}|^{\tau} + C  [u]^{\tau}_{W^{s, p}(A_{\ell} \cup A_{\ell+1})}  .
\end{equation*}
For simplicity, let  $a_{\ell} = \sum_{i=1}^{\sigma_{\ell}} |(u)_{A^{i}_{\ell}}|^{\tau}$. Then, the above inequality will become
\begin{equation*}
    2^{\ell (d-\alpha)} a_{\ell} \leq c_{1} 2^{(\ell+1)(d-\alpha)} a_{\ell+1} + C  [u]^{\tau}_{W^{s, p}(A_{\ell} \cup A_{\ell+1})}  .
\end{equation*}
Summing the above inequality from  $\ell=m \in \mathbb{Z}^{-}$ to  $-2$, we get
\begin{equation*}
   \sum_{\ell=m}^{-2} 2^{\ell (d-\alpha)} a_{\ell} \leq \sum_{\ell=m}^{-2} c_{1} 2^{(\ell+1)(d-\alpha)} a_{\ell+1} + C  \sum_{\ell=m}^{-2} [u]^{\tau}_{W^{s, p}(A_{\ell} \cup A_{\ell+1})}  .
\end{equation*}
By changing sides, rearranging, and re-indexing, we get
\begin{equation*}
   2^{m(d-\alpha)} a_{m} + (1-c_{1}) \sum_{\ell=m+1}^{-2} 2^{\ell (d-\alpha)} a_{\ell}  \leq  C a_{-1} 
    +  C \sum_{\ell=m}^{-2} [u]^{\tau}_{W^{s, p}(A_{\ell} \cup A_{\ell+1})}.
\end{equation*}
Adding  $2^{(\alpha -d)}a_{-1}$ on  both sides and choosing  $-m$ large enough such that  $|(u)_{A^{j}_{m}}|=0$ for all  $j \in \{ 1, \dots, \sigma_{m} \}$, we have, for some constant  $C= C(d,p,s, \tau,k)>0$,
\begin{equation*}
    (1-c_{1}) \sum_{\ell=m}^{-1} 2^{\ell (d-\alpha)} a_{\ell} \leq C a_{-1} +  C \sum_{\ell=m}^{-2} [u]^{\tau}_{W^{s, p}(A_{\ell} \cup A_{\ell+1})}  .
\end{equation*}
Putting the value of  $a_{\ell}$ in the above inequality, we obtain
\begin{equation}\label{eqnn99}
 (1-c_{1})   \sum_{\ell=m}^{-1} 2^{\ell (d-\alpha)} \sum_{i=1}^{\sigma_{\ell}} |(u)_{A^{i}_{\ell}}|^{\tau} \leq C \sum_{j=1}^{\sigma_{-1}} |(u)_{A^{j}_{-1}}|^{\tau} + C \sum_{\ell=m}^{-2} [u]^{\tau}_{W^{s, p}(A_{\ell} \cup A_{\ell+1})}  .
\end{equation}
Combining  \eqref{eqnn1} and  \eqref{eqnn99} yields
\begin{equation}\label{eqqnnn}
    \sum_{\ell=m}^{-1} \int_{A_{\ell}} \frac{|u(x)|^{\tau}}{|x_{k}|^{\alpha} } \,  dx \leq  C \sum_{j=1}^{\sigma_{-1}} |(u)_{A^{j}_{-1}}|^{\tau} +  C \sum_{\ell=m}^{-2} [u]^{\tau}_{W^{s, p}(A_{\ell} \cup A_{\ell+1})}  .
\end{equation}
Using H$\ddot{\text{o}}$lder's inequality with  $ \frac{1}{p} +  \frac{1}{p'} = 1$, we have
\begin{equation*}
      \sum_{j=1}^{\sigma_{-1}} |(u)_{A^{j}_{-1}}|^{\tau} =   \sum_{j=1}^{\sigma_{-1}} \Bigg| \frac{1}{|A^{j}_{-1}|} \int_{A^{j}_{-1}} u(x) \, dx \Bigg|^{\tau}
     \leq  C \sum_{j=1}^{\sigma_{-1}} \Bigg| \int_{A^{j}_{-1}} |u(x)|^{p} \, dx \Bigg|^{\frac{\tau}{p}}.
\end{equation*}
Now applying  \eqref{sumineq} with  $\gamma= \frac{\tau}{p}$, we obtain
\begin{equation*}
    \sum_{j=1}^{\sigma_{-1}} \Bigg| \int_{A^{j}_{-1}} |u(x)|^{p} \, dx \Bigg|^{\frac{\tau}{p}} \leq   \Bigg| \sum_{j=1}^{\sigma_{-1}}\int_{A^{j}_{-1}} |u(x)|^{p} \, dx \Bigg|^{\frac{\tau}{p}}  \leq 
    \Bigg| \int_{\Omega_{n}} |u(x)|^{p} \, dx \Bigg|^{\frac{\tau}{p}} \leq  \|u\|^{\tau}_{L^{p}(\Omega_{n})}.
\end{equation*}
Therefore, we have
\begin{equation*}
     \sum_{\ell=m}^{-1} \int_{A_{\ell}} \frac{|u(x)|^{\tau}}{|x_{k}|^{\alpha} } \, dx \leq  C \left( \|u\|^{\tau}_{L^{p}(\Omega_{n})} +  \sum_{\ell=m}^{-2} [u]^{\tau}_{W^{s, p}(A_{\ell} \cup A_{\ell+1})} \right)  .
\end{equation*}
Applying  \eqref{sumineq} with  $\gamma = \frac{\tau}{p}$ and from  \eqref{se1} (see Appendix  \ref{appendix}), we have
\begin{equation}\label{eqnn 111}
    \sum_{\ell=m}^{-2} [u]^{\tau}_{W^{s, p}(A_{\ell} \cup A_{\ell+1})} \leq  \left( \sum_{\ell=m}^{-2} [u]^{p}_{W^{s, p}(A_{\ell} \cup A_{\ell+1})} \right)^{\frac{\tau}{p}} \leq 2^{\frac{\tau}{p}} [u]^{\tau}_{W^{s,p}(\Omega_{n})}.
\end{equation}
Combining the above two inequalities, and using the fact that  $\tau \geq p$, we obtain
 \begin{equation*}
      \left( \bigintssss_{\Omega_{n}} \frac{|u(x)|^{\tau}}{|x_{k}|^{\alpha} } \,  dx \right)^{\frac{1}{\tau}} \leq C \left( [u]^{\tau}_{W^{s, p}(\Omega_{n})} + \|u\|^{\tau}_{L^{p}(\Omega_{n})} \right)^{\frac{1}{\tau}}  
       \leq C \left( [u]^{p}_{W^{s, p}(\Omega_{n})} + \|u\|^{p}_{L^{p}(\Omega_{n})} \right)^{\frac{1}{p}} .
 \end{equation*}
 This finishes the proof of the lemma.
\end{proof}

The next lemma proves Theorem  \ref{theorem 3 sp>k} and establishes the fractional Hardy-type inequality for the case  $sp>k$, where $1<k<d, ~ k \in \mathbb{N}$ when the domain is  $\Omega_{n}$ and $K= \{ (x_{k}, x_{d-k}) \in \mathbb{R}^{d} : x_{k} = 0 \ \text{and} \ x_{d-k} \in [-n,n]^{d-k} \}$. 

\begin{lemma}\label{flat case sp>k a}
     Let  $ 1<k<d$ with $k \in \mathbb{N}$ and  $sp>k$. For any  $\tau \geq p$ when  $sp=d$,  and  $\tau \in \left[p,p^{*}_{s} - \frac{kp}{d-sp} \right) $ when  $sp<d$, and  $x= (x_{k}, x_{d-k}) \in \Omega_{n}$, with $x_{k} \in \mathbb{R}^{k}$, there exists a constant  $C= C(d,p,s,\tau, k) > 0$  such that
    \begin{equation}
     \left(  \bigintssss_{\Omega_{n}} \frac{|u(x)|^{\tau}}{|x_{k}|^{\alpha} } \,  dx \right)^{\frac{1}{\tau}} \leq  C 
 [u]_{W^{s, p}(\Omega_{n})}, \quad \forall \ u \in W^{s,p}_{0}(\Omega_n \setminus K).
    \end{equation}
\end{lemma}
\begin{proof}
Let  $u \in W^{s,p}(\Omega_{n}) \cap C(\Omega_{n})$ be such that  $u=0$ in a neighbourhood of  $K$. For any  $\tau \in\left[p,p^{*}_{s} - \frac{kp}{d-sp} \right) $ and  $sp<d$, we have  $\alpha \in (k, sp]$. Let  $A^{j}_{\ell+1}$ be such that it satisfies  \eqref{codn on Al and Al+1}. Using triangle inequality, we have 
\begin{equation*}
    |(u)_{A^{j}_{\ell+1}}|^\tau \leq  \left( |(u)_{A^{i}_{\ell}}| + |(u)_{A^{i}_{\ell}} - (u)_{A^{j}_{\ell+1}}| \right)^\tau.
\end{equation*}
For  $1<k<d$ with $k \in \mathbb{N}$ and  $sp>k$, applying Lemma  \ref{estimate} with  $c :=c_{1}2^{\alpha-k}>1$, where  $c_{1} = \frac{2}{1+ 2^{\alpha-k}} <1$ together with Lemma  \ref{est2}, we obtain
\begin{equation*}
    |(u)_{A^{j}_{\ell+1}}|^{\tau} \leq c_{1} 2^{\alpha-k} |(u)_{A^{i}_{\ell}}|^{\tau} + C  2^{\ell(sp-d) \frac{\tau}{p}} [u]^{\tau}_{W^{s, p}(A^{i}_{\ell} \cup A^{j}_{\ell+1})}  .
\end{equation*}
Multiplying the above inequality by  $2^{(\ell+1)(d-\alpha)-d+k}$, and using the definition of  $\alpha$, we get
\begin{equation*}
   2^{(\ell+1)(d-\alpha)-d+k} |(u)_{A^{j}_{\ell+1}}|^{\tau} \leq c_{1} 2^{\ell (d-\alpha)}|(u)_{A^{i}_{\ell}}|^{\tau} + C [u]^{\tau}_{W^{s, p}(A^{i}_{\ell} \cup A^{j}_{\ell+1})}  .
\end{equation*}
Since there are  $2^{d-k}$ such  $A^{i}_{\ell}$'s that satisfy \eqref{codn on Al and Al+1}, summing the above inequality from  $i=2^{d-k}(j-1)+1$ to  $2^{d-k}j$, we obtain
\begin{align*}
     2^{d-k}2^{(\ell+1)(d-\alpha)-d+k}  |(u)_{A^{j}_{\ell+1}}|^{\tau}  & \leq c_{1} 2^{\ell (d-\alpha)} \sum_{i=2^{d-k}(j-1)+1}^{2^{d-k}j}  |(u)_{A^{i}_{\ell}}|^{\tau} \\ & \quad
       + \ C  \sum_{i=2^{d-k}(j-1)+1}^{2^{d-k}j} [u]^{\tau}_{W^{s, p}(A^{i}_{\ell} \cup A^{j}_{\ell+1})}  .
\end{align*}
Again, summing the above inequality from  $j=1$ to  $\sigma_{\ell+1}$, and using  \eqref{summation relation}, we obtain
\begin{align}\label{ineqn1}
2^{(\ell+1)(d-\alpha)}  \sum_{j=1}^{\sigma_{\ell+1}}  |(u)_{A^{j}_{\ell+1}}|^{\tau} & \leq c_{1} 2^{\ell (d-\alpha)} 
\sum_{i=1}^{\sigma_{\ell}} |(u)_{A^{i}_{\ell}}|^{\tau} \nonumber \\ & \quad
  + \ C \sum_{j=1}^{\sigma_{\ell+1}} \Bigg( \sum_{i=2^{d-k}(j-1)+1}^{2^{d-k}j}  [u]^{\tau}_{W^{s, p}(A^{i}_{\ell} \cup A^{j}_{\ell+1})} \Bigg).
\end{align}
From  \eqref{ineqn22}, we have
\begin{equation*}
     \sum_{j=1}^{\sigma_{\ell+1}} \Bigg( \sum_{i=2^{d-k}(j-1)+1}^{2^{d-k}j}  [u]^{\tau}_{W^{s, p}(A^{i}_{\ell} \cup A^{j}_{\ell+1})} \Bigg) \leq [u]^{\tau}_{W^{s, p}(A_{\ell} \cup A_{\ell+1})} .
\end{equation*}
Therefore, we obtain
\begin{equation*}
2^{(\ell+1)(d-\alpha)}  \sum_{j=1}^{\sigma_{\ell+1}}  |(u)_{A^{j}_{\ell+1}}|^{\tau} \leq c_{1} 2^{\ell (d-\alpha)} 
\sum_{i=1}^{\sigma_{\ell}} |(u)_{A^{i}_{\ell}}|^{\tau} + C  [u]^{\tau}_{W^{s, p}(A_{\ell} \cup A_{\ell+1})}  .
\end{equation*}
Summing the above inequality from  $\ell=m \in \mathbb{Z}^{-}$ to  $-2$, we get
\begin{equation*}
   \sum_{\ell=m}^{-2} 2^{(\ell+1)(d-\alpha)}  \sum_{j=1}^{\sigma_{\ell+1}}  |(u)_{A^{j}_{\ell+1}}|^{\tau} \leq c_{1} \sum_{\ell=m}^{-2}  2^{\ell (d-\alpha)} \sum_{i=1}^{\sigma_{\ell}} |(u)_{A^{i}_{\ell}}|^{\tau} + C  \sum_{\ell=m}^{-2} [u]^{\tau}_{W^{s, p}(A_{\ell} \cup A_{\ell+1})}  .
\end{equation*}
By changing sides, rearranging, and re-indexing, we get
\begin{equation*}
 (1-c_{1})   \sum_{\ell=m+1}^{-1} 2^{\ell (d-\alpha)} \sum_{i=1}^{\sigma_{\ell}} |(u)_{A^{i}_{\ell}}|^{\tau} \leq 2^{m(d- \alpha)} \sum_{j=1}^{\sigma_{m}} |(u)_{A^{j}_{m}}|^{\tau} + C \sum_{\ell=m}^{-2} [u]^{\tau}_{W^{s, p}(A_{\ell} \cup A_{\ell+1})}  .
\end{equation*}
Choose  $-m$ large enough such that  $|(u)_{A^{j}_{m}}|=0$ for all  $j \in \{ 1, \dots, \sigma_{m} \}$. Therefore, we have
\begin{equation}\label{eqnn}
    (1-c_{1})   \sum_{\ell=m}^{-1} 2^{\ell (d-\alpha)} \sum_{i=1}^{\sigma_{\ell}} |(u)_{A^{i}_{\ell}}|^{\tau} \leq C \sum_{\ell=m}^{-2} [u]^{\tau}_{W^{s, p}(A_{\ell} \cup A_{\ell+1})}  .
\end{equation}
Combining  \eqref{eqnn1} and  \eqref{eqnn} yields
\begin{equation*}
    \sum_{\ell=m}^{-1} \int_{A_{\ell}} \frac{|u(x)|^{\tau}}{|x_{k}|^{\alpha} } \, dx \leq   C \sum_{\ell=m}^{-2} [u]^{\tau}_{W^{s, p}(A_{\ell} \cup A_{\ell+1})}  .
\end{equation*}
Therefore, using  \eqref{eqnn 111}, we have
 \begin{equation*}
      \left( \int_{\Omega_{n}} \frac{|u(x)|^{\tau}}{|x_{k}|^{\alpha} } \,  dx \right)^{\frac{1}{\tau}} \leq C [u]_{W^{s, p}(\Omega_{n})}.
 \end{equation*}
 This proves the lemma.
\end{proof}

\smallskip

The following lemma proves Theorem \ref{theorem 1 sp=k} for the case $sp = k = d$. It establishes a fractional Hardy inequality with a singularity at the origin in a ball of radius $R > 0$ centered at $0$. This lemma also helps in obtaining a fractional Hardy inequality with a singularity at the origin in any bounded domain $\Omega$ containing $0$.

\begin{lemma}\label{flat case sp=k=d}
    Let  $d \geq1$ and  $sp=d$. Then  for any  $\tau \geq p$ and  $u \in W^{s,p}_{0}(B_{R}(0) \setminus \{0\})$ for some  $R>0$, we have
    \begin{equation}
    \left( \bigintsss_{B_{R}(0)} \frac{|u(x)|^{\tau}}{|x|^{d} \ln^{\tau} \left( \frac{2R}{|x|} \right) }  \, dx \right)^{\frac{1}{\tau}} \leq  C \|u\|_{W^{s,p}(B_{R}(0))}  ,
    \end{equation}
     where  $C= C(d,p, \tau, R)$ is a positive constant.  
\end{lemma}
\begin{proof}
   Let  $u \in W^{s,p}(B_{R}(0)) \cap C(B_{R}(0))$ be such that  $u=0$ in a neighbourhood of  $\{ 0 \}$. For each  $\ell \leq -1$, set
    \begin{equation*}
        A_{\ell} = \{ x \in \mathbb{R}^d : 2^{\ell} R \leq |x| < 2^{\ell+1}R \} \ .
    \end{equation*}
    Then, we have
    \begin{equation*}
        B_{R}(0) = \bigcup_{\ell=-\infty}^{-1} A_{\ell} \ .
    \end{equation*}
    Applying Lemma  \ref{sobolev} with  $\Omega = \{ x \in \mathbb{R}^d : ~ R < |x|< 2R \}, ~ \lambda= 2^{\ell}$ and  $sp=d$, we have
    \begin{equation*}
         \left( \fint_{A_{\ell}} |u(x)-(u)_{A_{\ell}}|^{\tau }  \,  dx \right)^{\frac{1}{\tau}} \leq C  [u]_{W^{s,p}(A_{\ell})}   ,
    \end{equation*}
  where  $C=C(d,p, \tau, R)$ is a positive constant. Let  $x \in A_{\ell}$, which imply  $ \frac{1}{|x|} < \frac{1}{2^{\ell}R}$. Therefore, we have
    \begin{align*}
         \int_{A_{\ell}} \frac{|u(x)|^{\tau}}{|x|^{d}} \, dx & \leq \frac{C}{2^{\ell d}} \int_{A_{\ell}} |u(x)-(u)_{A_{\ell}} + (u)_{A_{\ell}}|^{\tau} \, dx 
  \\ &  \leq \frac{C}{2^{\ell d}} \int_{A_{\ell}} |u(x)-(u)_{A_{\ell}}|^{\tau}  \, dx  + \frac{C}{2^{\ell d}} \int_{A_{\ell}} |(u)_{A_{\ell}}|^{\tau} \, dx  \\ &
    = C \frac{|A_{\ell}|}{2^{\ell d}} \fint_{A_{\ell}} |u(x)-(u)_{A_{\ell}}|^{\tau} \, dx + 
   C \frac{|A_{\ell}|}{2^{\ell d}}  |(u)_{A_{\ell}}|^{\tau}  \\ &
    \leq C [u]^{\tau}_{W^{s,p}(A_{\ell})} + C  |(u)_{A_{\ell}}|^{\tau}  ,
    \end{align*}
where $C= C(d,p, \tau, R)>0$. For each  $x \in A_{\ell}$, we have  $ \frac{2R}{|x|}  >  2^{- \ell}$. This implies that  $\ln \left( \frac{2R}{|x|} \right) > (- \ell) \ln2$. Therefore, we have
\begin{equation*}
    \bigintsss_{A_{\ell}} \frac{|u(x)|^{\tau}}{|x|^{d} \ln^{\tau}\left( \frac{2R}{|x|} \right)} \, dx \leq C \frac{[u]^{\tau}_{W^{s,p}(A_{\ell})} }{(- \ell)^{\tau}} + C \frac{|(u)_{A_{\ell}}|^{\tau}}{(-\ell)^{\tau}} 
     \leq  C [u]^{\tau}_{W^{s,p}(A_{\ell})} + C \frac{|(u)_{A_{\ell}}|^{\tau}}{(-\ell)^{\tau}}   .\end{equation*}
Here, we have used the fact that  $\frac{1}{(-\ell)^{\tau}} \leq 1$. Summing the above inequality from  $\ell=m \in \mathbb{Z}^{-}$ to  $-1$, we get
\begin{equation}\label{spd1}
    \bigintsss_{\{ 2^{m}R \leq |x| < R \} }  \frac{|u(x)|^{\tau}}{|x|^{d} \ln^{\tau}\left( \frac{2R}{|x|} \right)} \,  dx \leq C \sum_{\ell=m}^{-1} [u]^{\tau}_{W^{s,p}(A_{\ell})} + C   \sum_{\ell=m}^{-1} \frac{|(u)_{A_{\ell}}|^{\tau}}{(-\ell)^{\tau}}  .
\end{equation} 
Using Lemma  \ref{avg} with  $E= A_{\ell}$ and  $F=A_{\ell+1}$, and Lemma  \ref{sobolev} with  $\Omega = \{ x \in \mathbb{R}^d : R < |x| < 2^{2}R \}$ and  $\lambda=2^{\ell}$, we obtain, for some constant $C$ 
(independent of $\ell$), that
\begin{equation*}
   |(u)_{A_{\ell}} - (u)_{A_{\ell+1}} |^{\tau} \leq C \fint_{A_{\ell} \cup A_{\ell+1} }  |u(x) - (u)_{A_{\ell} \cup A_{\ell+1}}|^{\tau} \,  dx 
   \leq C [u]^{\tau}_{W^{s,p}(A_{\ell} \cup A_{\ell+1})} .
\end{equation*}
Independently, using triangle inequality, we have 
\begin{equation*}
    |(u)_{A_{\ell}}|^\tau \leq (|(u)_{A_{\ell+1}}| + |(u)_{A_{\ell}} - (u)_{A_{\ell+1}}|)^\tau.
\end{equation*}
For  $\ell  \leq -1$, consider the number  $c := (\frac{-\ell }{-\ell-1/2})^{\tau-1}>1$. Now applying Lemma  \ref{estimate} with  $c$, we obtain
\begin{equation*}
    |(u)_{A_{\ell}}|^{\tau} \leq \left( \frac{-\ell}{-\ell- 1/2} \right)^{\tau -1} |(u)_{A_{\ell+1}}|^{\tau} + C (-\ell)^{\tau-1} [u]^{\tau}_{W^{s,p}(A_{\ell} \cup A_{\ell+1})}.
\end{equation*}
Therefore, we have
\begin{equation*}
   \frac{|(u)_{A_{\ell}}|^{\tau}}{(-\ell)^{\tau-1}}  \leq  \frac{|(u)_{A_{\ell+1}}|^{\tau}}{(-\ell- 1/2)^{\tau-1}}  + C  [u]^{\tau}_{W^{s,p}(A_{\ell} \cup A_{\ell+1})}.
\end{equation*}
Summing the above inequality from  $\ell=m \in \mathbb{Z}^{-}$ to  $-2$, we get
\begin{equation*}
     \sum_{\ell=m}^{-2} \frac{|(u)_{A_{\ell}}|^{\tau}}{(-\ell)^{\tau-1}}  \leq \sum_{\ell=m}^{-2}  \frac{|(u)_{A_{\ell+1}}|^{\tau}}{(-\ell- 1/2)^{\tau-1}}  + C \sum_{\ell=m}^{-2}  [u]^{\tau}_{W^{s,p}(A_{\ell} \cup A_{\ell+1})}.
\end{equation*}
By changing sides, rearranging, and re-indexing, we get
\begin{align*}
   \frac{|(u)_{A_{m}}|}{(-m)^{\tau-1}} + \sum_{\ell=m+1}^{-2} \left\{  \frac{1}{(-\ell)^{\tau-1}} -  \frac{1}{(-\ell+ 1/2)^{\tau-1}}  \right\} |(u)_{A_{\ell}}|^{\tau} & \leq C|(u)_{A_{-1}}|^{\tau} \\ & \quad + C \sum_{\ell=m}^{-2}  [u]^{\tau}_{W^{s,p}(A_{\ell} \cup A_{\ell+1})}.
\end{align*}
Now, for large values of  $-\ell$, using the asymptotics 
\begin{equation*}
    \frac{1}{(-\ell)^{\tau-1}} -\frac{1}{(-\ell+1/2)^{\tau-1}} \sim \frac{1}{(-\ell)^{\tau}},
\end{equation*}
and choosing large $-m$ such that  $|(u)_{A_{m}}| = 0$, we arrive at
\begin{equation*}
    \sum_{\ell=m}^{-2} \frac{|(u)_{A_{\ell}}|^{\tau}}{(-\ell)^{\tau}} \leq C |(u)_{A_{-1}}|^{\tau} + C \sum_{\ell =m}^{-2} [u]^{\tau}_{W^{s,p}(A_{\ell} \cup A_{\ell+1})}.
\end{equation*}
Also, 
\begin{equation*}
     |(u)_{A_{-1}}|^{\tau} =   \left|  \frac{1}{|A_{-1}|} \int_{A_{-1}} u(x) \, dx  \Big|^{\tau} \leq C \right| \int_{A_{-1}} |u(x)|^{p} \, dx \Big|^{\frac{\tau}{p}} \\ \leq C \|u\|^{\tau}_{L^{p}(B_{R}(0))}.
\end{equation*} 
Therefore, we have
\begin{equation}\label{spd2}
     \sum_{\ell=m}^{-2} \frac{|(u)_{A_{\ell}}|^{\tau}}{(-\ell)^{\tau}} \leq  C \sum_{\ell =m}^{-2} [u]^{\tau}_{W^{s,p}(A_{\ell} \cup A_{\ell+1})} + C \|u\|^{\tau}_{L^{p}(B_{R}(0))}.
\end{equation}
Combining  \eqref{spd1} and  \eqref{spd2}, we obtain
\begin{equation*}
      \bigintsss_{\{ 2^{m}R \leq |x| < R \} }  \frac{|u(x)|^{\tau}}{|x|^{d} \ln^{\tau}\left( \frac{2R}{|x|} \right)} \, dx \leq C \sum_{\ell =m}^{-2} [u]^{\tau}_{W^{s,p}(A_{\ell} \cup A_{\ell+1})} +  C \|u\|^{\tau}_{L^{p}(B_{R}(0))} .
\end{equation*}
Applying  \eqref{sumineq} with  $\gamma = \frac{\tau}{p}$ on the first term of right hand side of the above inequality and then from  \eqref{se1} (see Appendix  \ref{appendix}) with  $\Omega_{n}= B_{R}(0)$, we get
\begin{align*}
   \left( \bigintsss_{B_{R}(0)} \frac{|u(x)|^{\tau}}{|x|^{d} \ln^{\tau} \left( \frac{2R}{|x|} \right) } \, dx \right)^{\frac{1}{\tau}} & \leq C \left( [u]^{\tau}_{W^{s,p}(B_{R}(0))} + \|u\|^{\tau}_{L^{p}(B_{R}(0))} \right)^{\frac{1}{\tau}}  \\ & \leq  C \left( [u]^{p}_{W^{s,p}(B_{R}(0))} + \|u\|^{p}_{L^{p}(B_{R}(0))} \right)^{\frac{1}{p}}  .  
\end{align*}
In the last inequality, we have used the fact that  $\tau \geq p$.
\end{proof}


\section{Proof of the main results}\label{proof of main result}
In this section, we prove Theorem  \ref{theorem 1 sp=k} without the optimality of weight function, Theorem  \ref{theorem 2 sp<k} and Theorem  \ref{theorem 3 sp>k}. Let $\Omega$ be a bounded Lipschitz domain in $\mathbb{R}^{d}, ~   d \geq 2$, and let  $K \subset \Omega$ be a compact set of codimension $k$ of class $C^{0,1}$, where $1<k<d,~   k \in \mathbb{N}$. For each $x \in K$ there exists a ball $B_{r_{x}}(x), ~ r_{x}>0$ such that Definition  \ref{definition} hold with an isomorphism $T_{x}$. Then  $ K \subset \cup_{x \in  K} B_{r_{x}}(x)$. Since  $K$ is compact, there exists  $x_{1}, \dots, x_{n} \in  K$ such that
\begin{equation*}
    K  \subset \bigcup_{i=1}^{n} B_{r_{i}}(x_{i}) ,
\end{equation*}
where  $r_{x_{i}}= r_{i}$. 


\subsection{Proof of Theorem  \ref{theorem 1 sp=k} without optimality of the weight function} Let  $u \in W^{s,p}(\Omega) \cap C(\Omega)$ be such that  $u=0$ in a neighbourhood of  $K$, and $sp=k$, where  $1<k<d$ with $k \in \mathbb{N}$.  Let  $\Omega \setminus K  \subset \cup_{i=0}^{n} \Omega_{i}$, where  $\Omega_{0} \subset \Omega \setminus K$ such that $\delta_{K}(x)> R_{0}$ for all $x \in \Omega_{0}$, for some $R_{0}>0$, and  $\Omega_{i} = B_{r_{i}}(x_{i})$ for all  $ 1 \leq i \leq n $. Let  $ \{ \eta_{i} \}_{i=0}^{n}$ be the associated partition of unity. Then,
\begin{equation*}
    u = \sum_{i=0}^{n} u_{i}, \hspace{.3cm} \text{where} \  u_{i} = \eta_{i} u.
\end{equation*}
From Lemma  \ref{testfunc}, we have
\begin{equation*}
    \|u_{i}\|_{W^{s,p}(\Omega)} \leq C \|u\|_{W^{s,p}(\Omega)}, \hspace{3mm} \forall \ 0 \leq i \leq n .
\end{equation*}
Therefore, it is sufficient to prove Theorem  \ref{theorem 1 sp=k} for all  $u_{i}, ~ 0 \leq i \leq n$. Since  $\operatorname{supp} u_{0} \subset \Omega_{0}$ and, for all  $x \in \Omega_{0}$, 
 \begin{equation*}
     C_{1} \leq \delta_{K}(x) \leq C_{2} \quad \text{for some} \  C_{1}, C_{2} >0 .
 \end{equation*}
 Therefore, from fractional Sobolev inequality  \eqref{Sobineqq}, we have
\begin{equation*} 
 \left( \bigintsss_{\Omega_{0}} \frac{|u_{0}(x)|^{\tau}}{\delta_{K}^{\alpha}(x) \ln^{\beta}\left( \frac{2R}{ \delta_{K}(x)}\right)} \,  dx \right)^{\frac{1}{\tau}} \leq C \Bigg( \int_{\Omega_{0}} |u_{0}(x)|^{\tau} \, dx \Bigg)^{\frac{1}{\tau}} \leq C \|u_{0}\|_{W^{s,p}(\Omega_{0})} .
\end{equation*}
For  $1 \leq i \leq n$, we have  $ \operatorname{supp} u_{i} \subset  (\Omega \backslash K) \cap \Omega_{i} $. Consider the isomorphism $T_{x_{i}}$, then 
 \begin{equation*}
      \delta_{K}(x) \sim |\xi_{k}| \quad \text{for all} \  x \in (\Omega \setminus K) \cap \Omega_{i},
 \end{equation*}
where  $T_{x_{i}}((\xi_{k}, \xi_{d-k})) = x$. Therefore, from Lemma  \ref{flat case sp=k 2}, we have
 \begin{align*}
    \left( \bigintsss_{(\Omega \setminus K)  \cap \Omega_{i}} \frac{|u_{i}(x)|^{\tau}}{\delta_{K}^{\alpha}(x) \ln^{\beta}\left( \frac{2R}{\delta_{K}(x)} \right)} \, dx \right)^{\frac{1}{\tau}} & \sim  \left( \bigintsss_{T^{-1}_{x_{i}}((\Omega \setminus K) \cap \Omega_{i})} \frac{|u_{i} \circ T_{x_{i}}(\xi)|^{\tau}}{|\xi_{k}|^{\alpha} \ln^{\beta}\left( \frac{2R}{|\xi_{k}|} \right)} \, d \xi \right)^{\frac{1}{\tau}} \\ & \leq C \|u_{i} \circ T_{x_{i}}\|_{W^{s,p}(T^{-1}_{x_{i}}((\Omega \setminus K) \cap \Omega_{i}))} \\ & = C \|u_{i}\|_{W^{s,p}((\Omega \setminus K) \cap \Omega_{i})}.
 \end{align*}
This proves Theorem  \ref{theorem 1 sp=k} without optimality of weight function when  $1<k<d$ with $k \in \mathbb{N}$.

\smallskip

Now assume $k=d$ and $0 \in \Omega$. There exists $R_{1}>0$ such that $B_{R_{1}}(0) \subset \Omega$. From Lemma  \ref{flat case sp=k=d}, we have
\begin{equation*}
    \left( \bigintsss_{B_{R_{1}}(0)} \frac{|u(x)|^{\tau}}{|x|^{d} \ln^{\tau} \left( \frac{2R_{1}}{|x|} \right) } \, dx \right)^{\frac{1}{\tau}} \leq  C \|u\|_{W^{s,p}(B_{R_{1}}(0))}  .
\end{equation*}
Also, for any $x \in \Omega \setminus B_{R_{1}}(0)$, we have $|x| \geq R_{1}$. Therefore, using this and fractional Sobolev inequality  \eqref{Sobineqq}, we have
\begin{equation*}
    \left( \bigintsss_{\Omega \setminus B_{R_{1}}(0)} \frac{|u(x)|^{\tau}}{|x|^{d} \ln^{\tau} \left( \frac{2R}{|x|} \right) } \,  dx \right)^{\frac{1}{\tau}} \leq C \left( \int_{\Omega \setminus B_{R_{1}}(0)} |u(x)|^{\tau} \, dx \right)^{\frac{1}{\tau}} \leq C\|u\|_{W^{s,p}(\Omega)}.
\end{equation*}
This proves the Theorem  \ref{theorem 1 sp=k} for the case $k=d$ and $K = \{ 0 \}$.

\subsection{Proof of Theorem  \ref{theorem 2 sp<k}}
 Let  $u \in W^{s,p}(\Omega) \cap C(\Omega)$ be such that $u=0$ in a neighbourhood of $K$ and $sp<k$, where  $ 1<k<d$ with $k \in \mathbb{N}$. Let  $\Omega \setminus K \subset \cup_{i=0}^{n} \Omega_{i}$, where  $\Omega_{0} \subset \Omega \setminus K$ such that $\delta_{K}(x)> R_{0}$ for all $x \in \Omega_{0}$ for some $R_{0}>0$, and  $\Omega_{i} = B_{r_{i}}(x_{i})$ for all  $ 1 \leq i \leq n $. Let  $ \{ \eta_{i} \}_{i=0}^{n}$ be the associated partition of unity. Then,
\begin{equation*}
    u = \sum_{i=0}^{n} u_{i} \quad \text{where} \  u_{i} = \eta_{i} u.
\end{equation*}
From Lemma  \ref{testfunc}, we have
\begin{equation*}
    \|u_{i}\|_{W^{s,p}(\Omega)} \leq C \|u\|_{W^{s,p}(\Omega)}, \hspace{3mm} \forall \ 0 \leq i \leq n .
\end{equation*}
Therefore, it is sufficient to prove Theorem  \ref{theorem 2 sp<k} for all  $u_{i}, ~ 0 \leq i \leq n$. Since  $\operatorname{supp} u_{0} \subset \Omega_{0}$ and, for all  $x \in \Omega_{0}$, 
 \begin{equation*}
     C_{1} \leq \delta_{K}(x) \leq C_{2} \quad \text{for some} \  C_{1}, C_{2} >0 .
 \end{equation*}
 Therefore, from fractional Sobolev inequality  \eqref{Sobineqq}, we have
\begin{equation*} 
 \left( \int_{\Omega_{0}} \frac{|u_{0}(x)|^{\tau}}{\delta_{K}^{\alpha}(x) } \, dx \right)^{\frac{1}{\tau}} \leq C \Bigg( \int_{\Omega_{0}} |u_{0}(x)|^{\tau} \, dx \Bigg)^{\frac{1}{\tau}} \leq C \|u_{0}\|_{W^{s,p}(\Omega_{0})} .
\end{equation*}
For  $1 \leq i \leq n$, we have $\operatorname{supp} u_{i} \subset (\Omega \backslash K ) \cap \Omega_{i} $. Consider the isomorphism $T_{x_{i}}$, then 
 \begin{equation*}
      \delta_{K}(x) \sim |\xi_{k}| \quad \text{for all} \  x \in (\Omega \setminus K) \cap \Omega_{i},
 \end{equation*}
where  $T_{x_{i}}((\xi_{k}, \xi_{d-k})) = x$. Therefore, from Lemma  \ref{flat case sp<k}, we have
 \begin{align*}
    \left( \int_{(\Omega \setminus K) \cap \Omega_{i}} \frac{|u_{i}(x)|^{\tau}}{\delta_{K}^{\alpha}(x) } \,  dx \right)^{\frac{1}{\tau}} & \sim  \left( \int_{T^{-1}_{x_{i}}((\Omega \setminus K ) \cap \Omega_{i})} \frac{|u_{i} \circ T_{x_{i}}(\xi)|^{\tau}}{|\xi_{k}|^{\alpha}} \, d \xi \right)^{\frac{1}{\tau}} \\ & \leq C \|u_{i} \circ T_{x_{i}}\|_{W^{s,p}(T^{-1}_{x_{i}}((\Omega \setminus K) \cap \Omega_{i}))} = C \|u_{i}\|_{W^{s,p}((\Omega \setminus K) \cap \Omega_{i})}.
 \end{align*}

\subsection{Proof of Theorem  \ref{theorem 3 sp>k}}\label{proof of theorem 3}
Let  $u \in W^{s,p}(\Omega) \cap C(\Omega)$ be such that $u=0$ in a neighbourhood of $K$ and $sp>k$, where $1<k<d$ with $k \in \mathbb{N}$. To prove Theorem  \ref{theorem 3 sp>k}, we will utilize  Lemma  \ref{flat case sp>k a}. For this, first we consider the case $\tau=p$ when $sp<d$, and $\tau \geq p$ when $sp=d$. Following a similar approach illustrated in the proof of Theorem  \ref{theorem 2 sp<k} and using Lemma  \ref{flat case sp>k a}, one can obtain the following inequality:
\begin{equation}\label{proof of theorem sp>k full norm}
\left(    \int_{\Omega  } \frac{|u(x)|^{\tau}}{\delta^{sp}_{K}(x) } \, dx \right)^{\frac{1}{\tau}} \leq C \|u\|_{W^{s,p}(\Omega)},
\end{equation}
for  $ \tau=p $ when  $sp<d$, and  $\tau \geq p$ when  $sp=d$. To prove Theorem  \ref{theorem 3 sp>k}, it is sufficient to prove for some constant  $C>0$,
\begin{equation}\label{sp>k fractional sobolev}
    \int_{\Omega} |u(x)|^{p} dx \leq C [u]^{p}_{W^{s,p}(\Omega)}, \quad \ \forall \ u \in W^{s,p}_{0}(\Omega \setminus K),
\end{equation}
when  $sp>k$. Assume this is not true. Suppose we have a sequence of functions  $\{ u_{n} \}_{n}$ in $W^{s,p}_{0} (\Omega \setminus K)$ defined on the domain  $\Omega$, such that  $\int_{\Omega} |u_{n}(x)|^{p} dx =1 $ for all  $n \in \mathbb{N}$ and  $[u_{n}]_{W^{s,p}(\Omega)} \to 0$ as  $n \to \infty$. Let  $\mathcal{F}= \{ u_{n}: ~ n \in \mathbb{N} \}$, then  $\mathcal{F}$ satisfies the conditions of Lemma  \ref{compactness} and hence  $\mathcal{F}$ is pre-compact in  $L^{p}(\Omega)$. Therefore, there exists a subsequence (which we will denote as  $\{ u_{n} \}_{n}$ again) such that   $u_{n} \to u$ pointwise almost everywhere as  $n \to \infty$ and  $\int_{\Omega} |u(x)|^{p} dx = 1$. Also, by Fatou's lemma we have
\begin{equation*}
    [u]_{W^{s,p}(\Omega)} \leq \liminf_{n \to \infty} [u_{n}]_{W^{s,p}(\Omega)} = 0 .
\end{equation*}
Therefore,  $u$ must be a constant function. As  $\int_{\Omega} |u(x)|^{p} dx =1 $, we have  $u(x) = |\Omega|^{\frac{-1}{p}}$ almost everywhere. From  \eqref{proof of theorem sp>k full norm} and  $sp>k$, we have
\begin{equation*}
    \infty = |\Omega|^{\frac{-1}{p}} \left( \int_{\Omega} \frac{1}{\delta^{sp}_{K}(x)} dx \right)^{\frac{1}{\tau}} \leq C\|u\|_{W^{s,p}(\Omega)} = C, 
\end{equation*}
which is a contradiction. Therefore, from  \eqref{proof of theorem sp>k full norm} and  \eqref{sp>k fractional sobolev}, we have
\begin{equation}\label{ineq frac Hardy sp>k a}
    \left(    \int_{\Omega} \frac{|u(x)|^{\tau}}{\delta^{sp}_{K}(x) } \, dx \right)^{\frac{1}{\tau}} \leq C [u]_{W^{s,p}(\Omega)},
\end{equation}
for  $ \tau=p $ when  $sp<d$, and  $\tau \geq p$ when  $sp=d$. Now assume $sp<d$ and $\tau \in (p,p^{*}_{s})$, there exists $\theta \in (0,1)$ such that $\tau = \theta p + (1-\theta) p^{*}_{s}$. Using H$\ddot{\text{o}}$lder's inequality and $a_{1}>0$, we obtain
\begin{equation*}
    \int_{\Omega} \frac{|u(x)|^{\tau}}{\delta^{a_{1}}_{K}(x)} \, dx =  \int_{\Omega} \frac{|u(x)|^{\theta p + (1- \theta)p^{*}_{s}}}{\delta^{a_{1}}_{K}(x)} \, dx \leq \left( \int_{\Omega} \frac{|u(x)|^{p}}{\delta^{\frac{a_{1}}{\theta}}_{K}(x)} \, dx \right)^{\theta} \left( \int_{\Omega} |u(x)|^{p^{*}_{s}} \, dx \right)^{1- \theta}.
\end{equation*}
Choose $a_{1}>0$ such that $\frac{a_{1}}{\theta} = sp$. Then using  \eqref{ineq frac Hardy sp>k a} with $\tau=p,  ~ sp<d$ and fractional Sobolev inequality  \eqref{Sobineqq} with  \eqref{sp>k fractional sobolev}, we have
\begin{equation*}
    \int_{\Omega} \frac{|u(x)|^{\tau}}{\delta^{a_{1}}_{K}(x)} \, dx \leq C [u]^{\theta p}_{W^{s,p}(\Omega)} [u]^{(1-\theta)p^{*}_{s}}_{W^{s,p}(\Omega)}= C[u]^{\tau}_{W^{s,p}(\Omega)}. 
\end{equation*}
From the definition of $\theta$, we obtain $a_{1} = d+ (sp-d) \frac{\tau}{p} = \alpha$. Therefore, combining all the above inequalities, we get
\begin{equation*}
    \left( \int_{\Omega} \frac{|u(x)|^{\tau}}{\delta^{\alpha}_{K}(x)} dx \right)^{\frac{1}{\tau}}  \leq C [u]_{W^{s,p}(\Omega)},
\end{equation*}
for $\tau \in [p, p^{*}_{s}]$ when $sp<d$, and $\tau \geq p$ when $sp=d$. Consider the case $\tau=p$ and $sp>k$. Using fractional Poincar\'e\ inequality  \eqref{poincare}, one can obtain that Lemma  \ref{sumineqlemma} and Lemma  \ref{est2} holds true for any $s \in (0,1)$ and $\tau=p$. Therefore, Lemma  \ref{flat case sp>k a} holds true for any $s \in (0,1)$ and $\tau=p$ satisfying $sp>k$. Therefore, following a similar approach as illustrated above for the case $sp \leq d$, one can obtain the following inequality for $sp>k$,
\begin{equation*}
    \left(    \int_{\Omega} \frac{|u(x)|^{p}}{\delta^{sp}_{K}(x) } \, dx \right)^{\frac{1}{p}} \leq C [u]_{W^{s,p}(\Omega)}.
\end{equation*}
This proves Theorem \ref{theorem 3 sp>k}.

\subsection{Proof of Corollary  \ref{cor 1}}\label{proof of cor 1} Let $K_{1}$ be any set such that $K_{1} \subset K$, where $K$ is a compact set of codimension $k$ of class $C^{0,1}$, where $ 1< k<d, ~  k \in \mathbb{N}$. Let $g(t) = t^{\alpha} \ln^{\beta} \left( \frac{2R}{t} \right)$ be a function defined on $t \geq 0$, where $\alpha, ~ \beta >0$. Then the function $g$ is increasing for any $t \in (0, 2R e^{- \frac{\beta}{\alpha}})$, i.e., for any $t_{1}, t_{2} \in \left( 0, 2R e^{- \frac{\beta}{\alpha}} \right)$ with $t_{1} \geq t_{2}$, we have $g(t_{1}) \geq g(t_{2})$. Using the function $g$ defined as above with $\delta_{K}(x), ~ \delta_{K_{1}}(x) \in  \left( 0, 2R e^{- \frac{\beta}{\alpha}} \right) $, we have
\begin{equation*}
    \delta^{\alpha}_{K_{1}}(x) \ln^{\beta} \left( \frac{2R}{\delta_{K_{1}}(x)} \right) \geq \delta^{\alpha}_{K}(x) \ln^{\beta} \left( \frac{2R}{\delta_{K}(x)} \right) .
\end{equation*}
Let $\Omega'= \{ x \in \Omega : \max\{ \delta_{K}(x) , \delta_{K_{1}}(x)  \} < 2Re^{- \frac{\beta}{\alpha}} \} $. Then, using the above inequality and $C_{1} \leq \delta_{K}(x) \leq C_{2}$ for all $x \in \Omega \setminus \Omega'$ for some $C_{1}, C_{2} >0$, we get
\begin{equation*}
\begin{split}
    \bigintsss_{\Omega} \frac{|u(x)|^{\tau}}{\delta^{\alpha}_{K_{1}}(x) \ln^{\beta} \left( \frac{2R}{\delta_{K_{1}}(x)}  \right)} \,  dx & =   \bigintsss_{\Omega'} \frac{|u(x)|^{\tau}}{\delta^{\alpha}_{K_{1}}(x) \ln^{\beta} \left( \frac{2R}{\delta_{K_{1}}(x)}  \right)} \, dx +  \bigintsss_{\Omega \setminus \Omega'} \frac{|u(x)|^{\tau}}{\delta^{\alpha}_{K_{1}}(x) \ln^{\beta} \left( \frac{2R}{\delta_{K_{1}}(x)}  \right)} \, dx \\ & \leq \bigintsss_{\Omega'} \frac{|u(x)|^{\tau}}{\delta^{\alpha}_{K}(x) \ln^{\beta} \left( \frac{2R}{\delta_{K}(x)}  \right)} \, dx + C \int_{\Omega \setminus \Omega'} |u(x)|^{\tau} \, dx \\ & \leq \bigintsss_{\Omega} \frac{|u(x)|^{\tau}}{\delta^{\alpha}_{K}(x) \ln^{\beta} \left( \frac{2R}{\delta_{K}(x)}  \right)} \, dx + C \int_{\Omega} |u(x)|^{\tau} \, dx .
\end{split}
\end{equation*}
Therefore, from Theorem  \ref{theorem 1 sp=k} with $K$ of codimension $k$ and fractional Sobolev inequality (see  \eqref{Sobineqq}), we obtain
\begin{equation*}
   \left(  \bigintsss_{\Omega} \frac{|u(x)|^{\tau}}{\delta^{\alpha}_{K_{1}}(x) \ln^{\beta} \left( \frac{2R}{\delta_{K_{1}}(x)}  \right)} \, dx \right)^{\frac{1}{\tau}} \leq C \|u\|_{W^{s,p}(\Omega)}.
\end{equation*}


\section{Optimality of Weight function in Theorem  \ref{theorem 1 sp=k}}\label{optimality section}
In this section, we prove the optimality of weight function presented in Theorem  \ref{theorem 1 sp=k}. To prove the optimality of weight function for Theorem  \ref{theorem 1 sp=k}, it is sufficient to establish for flat boundary case. For a general bounded Lipschitz domain, the result follows from a patching argument. We assume the domain  $\Omega = B^{k}_{1}(0) \times (0,1)^{d-k}$, where  $B^{k}_{1}(0) \subset \mathbb{R}^{k}$ is a ball of radius  $1$ centered at  $0$ and a compact set $K= \{ x= (x_{k}, x_{d-k}) \in \mathbb{R}^{d} : x_{k} = 0  \ \text{and} \ x_{d-k} \in [0,1]^{d-k} \}$. Then  $\delta_{K}(x)= |x_{k}|$ for all  $x \in \Omega \setminus K$, and let  $u_{k}: B^{k}_{1}(0) (\subset \mathbb{R}^{k}) \to \mathbb{R}$ be a function such that
\begin{equation}\label{fn u_k}
    u_{k}(x_{k}) = \begin{cases}
    \frac{\ln\left( \frac{2}{\epsilon} \right)}{\ln \left( \frac{2}{|x_{k}|} \right)}, & 0< |x_{k}| < \epsilon \\
        1, & \epsilon \leq |x_{k}| < 1.
    \end{cases}
\end{equation}

\smallskip

The next lemma establishes an inequality for the function  $u_{k}$ which plays a crucial role in establishing optimality of weight function given in Theorem  \ref{theorem 1 sp=k}.

\begin{lemma}\label{lemma on u_k}
    Let $1<k<d$ with $k \in \mathbb{N}$ and $sp=k$, and let  $u_{k}$ be the function defined in  \eqref{fn u_k}. Then there exist constants  $C_{1}, ~ C_{2}>0$ does not depend on $\epsilon$ such that
    \begin{equation*}
        \bigintsss_{B^{k}_{1}(0)} \frac{|u_{k}(x_{k}) - (u_{k})_{B^{k}_{1}(0)}|^{p}}{|x_{k}|^{k} \ln^{p} \left( \frac{2}{|x_{k}|}  \right)} \, dx_{k} = \frac{C_{1}}{\ln^{p -1} \left( \frac{2}{\epsilon} \right)} + O(\epsilon)   , 
    \end{equation*}
    and 
    \begin{equation*}
         [u_{k}]^{p}_{W^{s,p}(B^{k}_{1}(0))} \leq \frac{C_{2}}{\ln^{p-1} \left( \frac{2}{\epsilon} \right)}.
    \end{equation*}
\end{lemma}
\begin{proof}
Since
\begin{align*}
    (u_{k})_{B^{k}_{1}(0)} & =  \frac{1}{|B^{k}_{1}(0)|}  \int_{B^{k}_{1}(0)} u_{k}(x_{k}) \, dx_{k} \\ & =  \frac{1}{|B^{k}_{1}(0)|} \left( \ln\left( \frac{2}{\epsilon} \right) \int_{0 < |x_{k}| < \epsilon} \frac{1}{\ln \left( \frac{2}{|x_{k}|} \right)} \, dx_{k} + \int_{\epsilon \leq |x_{k}| <1} \, dx_{k} \right) \\ & =  k \ln\left( \frac{2}{\epsilon} \right) \int_{0}^{\epsilon} \frac{r^{k-1}}{\ln \left( \frac{2}{r} \right) } \, dr + (1-\epsilon)^{k} = 1 + O(\epsilon),
\end{align*}
it follows that
\begin{align*}
    \bigintsss_{B^{k}_{1}(0)}  \frac{ | u_{k}(x_{k}) -  (u_{k})_{B^{k}_{1}(0)} |^{p}}{|x_{k}|^{k} \ln^{p} \left( \frac{2}{|x_{k}|}  \right)} \,  dx_{k} &  = C \bigintsss_{0 < |x_{k}|< \epsilon} \frac{|u_{k}(x_{k}) - 1|^{p}}{|x_{k}|^{k} \ln^{p} \left( \frac{2}{|x_{k}|}  \right)} \, dx_{k} + O(\epsilon)\\ &  = C \bigintsss_{0 < |x_{k}|< \epsilon} \frac{\left|\ln\left( \frac{2}{\epsilon} \right)-\ln \left( \frac{2}{|x_{k}|} \right) \right|^{p}}{|x_{k}|^{k} \ln^{2p} \left( \frac{2}{|x_{k}|} \right) } \, dx_{k} + O(\epsilon) \\ &  = C \bigintsss_{0}^{\epsilon} \frac{|\ln\left( \frac{2}{\epsilon} \right)-\ln \left( \frac{2}{r} \right)|^{p}}{r \ln^{2p} \left( \frac{2}{r} \right) } \, dr + O(\epsilon).
\end{align*}
Applying the change of variable  $ \frac{\ln \left( \frac{2}{r} \right)}{ \ln \left( \frac{2}{\epsilon} \right)}  = \xi$ and using the fact that $\int_{1}^{\infty} \frac{|\xi-1|^{p}}{\xi^{2p}} d \xi < \infty$, we obtain
\begin{align*}
\bigintsss_{B^{k}_{1}(0)} \frac{|u_{k}(x_{k}) - (u_{k})_{B^{k}_{1}(0)}|^{p}}{|x_{k}|^{k} \ln^{p} \left( \frac{2}{|x_{k}|}  \right)} \, dx_{k} &   = \frac{C}{\ln^{p -1} \left( \frac{2}{\epsilon} \right)} \int_{1}^{\infty} \frac{|\xi - 1|^{p}}{\xi^{2 p}} \, d \xi +   O(\epsilon) \\ & = \frac{C}{\ln^{p -1} \left( \frac{2}{\epsilon} \right)} + O(\epsilon) .
\end{align*}
This proves the first part of lemma. Now, let us calculate the Gagliardo seminorm of  $u_{k}$ with  $sp=k$,
\begin{align}\label{opt eq gsk}
     [u_{k}]^{p}_{W^{s,p}(B^{k}_{1}(0))} & =  [u_{k}]^{p}_{W^{s,p}(B^{k}_{\epsilon}(0))} + 2 \int_{\epsilon<|y_{k}|<1} \int_{0 < |x_{k}|< \epsilon} \frac{|u_{k}(x_{k}) - u_{k}(y_{k})|^{p}}{|x_{k} - y_{k}|^{2k}} \, dx_{k} \, dy_{k} \nonumber \\ & =: J_{1} + J_{2}.
 \end{align}
 For  $J_{1}$, applying the change of variable in polar coordinate, we have
 \begin{align*}
     J_{1} & =   \int_{ 0 < |x_{k}|< \epsilon}  \int_{0< |y_{k}|< \epsilon} \frac{|u_{k}(x_{k}) - u_{k}(y_{k})|^{p}}{|x_{k} - y_{k}|^{2k}} \, dx_{k} \, dy_{k}  \\ & =  \ln^{p} \left( \frac{2}{\epsilon} \right) \bigintsss_{ 0 < |x_{k}|< \epsilon}  \bigintsss_{0< |y_{k}|< \epsilon} \frac{ \left| \ln \left( \frac{2}{|x_{k}|} \right) - \ln \left( \frac{2}{|y_{k}|} \right) \right|^{p}}{ \ln^{p} \left( \frac{2}{|x_{k}|} \right) \ln^{p} \left( \frac{2}{|y_{k}|} \right) |x_{k} - y_{k}|^{2k} } \, dx_{k} \, dy_{k} \\ & =  \ln^{p} \left( \frac{2}{\epsilon} \right) \bigintsss_{\mathbb{S}^{k-1}} \bigintsss_{\mathbb{S}^{k-1}} \bigintsss_{0}^{\epsilon}  \bigintsss_{0}^{\epsilon} \frac{|\ln \left( \frac{2}{r} \right) - \ln \left( \frac{2}{t} \right)|^{p}}{ \ln^{p} \left( \frac{2}{r} \right) \ln^{p} \left( \frac{2}{t} \right) |r \omega_{r}  - t \omega|^{2k} } r^{k-1} t^{k-1} \, dr \, dt \, d \omega_{r} \, d \omega .
 \end{align*}
 Using rotational invariance, we have
 \begin{multline*}
   \bigintsss_{\mathbb{S}^{k-1}} \bigintsss_{\mathbb{S}^{k-1}}  \bigintsss_{0}^{\epsilon}  \bigintsss_{0}^{\epsilon} \frac{|\ln \left( \frac{2}{r} \right) - \ln \left( \frac{2}{t} \right)|^{p}}{ \ln^{p} \left( \frac{2}{r} \right) \ln^{p} \left( \frac{2}{t} \right) |r \omega_{r}  - t \omega|^{2k} } r^{k-1} t^{k-1} \, dr \, dt \, d \omega_{r} \, d \omega \\ =  |\mathbb{S}^{k-1}| \bigintsss_{\mathbb{S}^{k-1}} \bigintsss_{0}^{\epsilon}  \bigintsss_{0}^{\epsilon} \frac{|\ln \left( \frac{2}{r} \right) - \ln \left( \frac{2}{t} \right)|^{p}}{ \ln^{p} \left( \frac{2}{r} \right) \ln^{p} \left( \frac{2}{t} \right) |r e_{1}  - t \omega|^{2k} } r^{k-1} t^{k-1} \, dr \, dt \, d \omega ,
 \end{multline*}
 where  $e_{1}= (1, 0, \dots, 0)$. Combining the above two inequalities, we obtain
 \begin{align*}
     J_{1} & =  |\mathbb{S}^{k-1}| \ln^{p} \left( \frac{2}{\epsilon} \right) \bigintsss_{\mathbb{S}^{k-1}} \bigintsss_{0}^{\epsilon}  \bigintsss_{0}^{\epsilon} \frac{|\ln \left( \frac{2}{r} \right) - \ln \left( \frac{2}{t} \right)|^{p}}{ \ln^{p} \left( \frac{2}{r} \right) \ln^{p} \left( \frac{2}{t} \right) |r e_{1}  - t \omega|^{2k} } r^{k-1} t^{k-1} \, dr \, dt \, d \omega \\ & =  |\mathbb{S}^{k-1}| \ln^{p} \left( \frac{2}{\epsilon} \right) \bigintsss_{\mathbb{S}^{k-1}}  \bigintsss_{0}^{\epsilon}  \bigintsss_{0}^{\epsilon} \left( \frac{|\ln \left( \frac{2}{r} \right) - \ln \left( \frac{2}{t} \right)|^{p}}{ \ln^{p} \left( \frac{2}{r} \right) \ln^{p} \left( \frac{2}{t} \right) \left( \frac{r}{t} + \frac{t}{r} -  2 \omega_{1} \right)^{k} } \right) \frac{1}{rt} \,  dr \, dt \, d \omega ,
 \end{align*}
 where  $\omega= (\omega_{1}, \dots, \omega_{k})$. Again, using change of variables  $\ln \left( \frac{2}{r} \right) = \xi$ and  $\ln \left( \frac{2}{t} \right) = \eta$ in  $J_{1}$, we have
\begin{equation*}
    J_{1} =   |\mathbb{S}^{k-1}| \ln^{p} \left( \frac{2}{\epsilon} \right) \int_{\mathbb{S}^{k-1}} \int_{\ln \left( \frac{2}{\epsilon} \right)}^{\infty}\int_{\ln \left( \frac{2}{\epsilon} \right)}^{\infty} \frac{|\xi - \eta|^{p}}{\xi^{p} \eta^{p} \left( e^{\eta - \xi} + e^{\xi- \eta} -2 \omega_{1} \right)^{k} } \, d \xi \, d \eta \, d \omega.
\end{equation*}
Assume  $\eta - \xi >0$, let $z = \eta - \xi$, and using symmetry, we obtain
\begin{align}\label{opt 1}
    J_{1} & \leq  2  |\mathbb{S}^{k-1}| \ln^{p} \left( \frac{2}{\epsilon} \right) \int_{\mathbb{S}^{k-1}} \int_{\ln \left( \frac{2}{\epsilon} \right)}^{\infty} \int_{0}^{\infty} \frac{z^{p}}{\xi^{p} (\xi + z)^{p} ( e^{z} + e^{-z} -2 \omega_{1} )^{k} } \, dz \, d \xi \,  d \omega \nonumber \\ & = 2  |\mathbb{S}^{k-1}| \ln^{p} \left( \frac{2}{\epsilon} \right) \int_{\mathbb{S}^{k-1}} \int_{\ln \left( \frac{2}{\epsilon} \right)}^{\infty} \int_{0}^{2} \frac{z^{p}}{\xi^{p} (\xi + z)^{p} ( e^{z} + e^{-z} -2 \omega_{1} )^{k} } \, dz \, d \xi \,  d \omega \nonumber \\ & \quad + 2  |\mathbb{S}^{k-1}| \ln^{p} \left( \frac{2}{\epsilon} \right) \int_{\mathbb{S}^{k-1}} \int_{\ln \left( \frac{2}{\epsilon} \right)}^{\infty} \int_{2}^{\infty} \frac{z^{p}}{\xi^{p} (\xi + z)^{p} ( e^{z} + e^{-z} -2 \omega_{1} )^{k} } \, dz \, d \xi \,  d \omega  .
\end{align}
Using  $\frac{z^{p}}{(e^{z}+ e^{-z}-2 \omega_{1})^{k}} \in L^{1}((2, \infty))$, we have
\begin{multline}\label{opt 2}
   \ln^{p} \left( \frac{2}{\epsilon} \right) \int_{\mathbb{S}^{k-1}} \int_{\ln \left( \frac{2}{\epsilon} \right)}^{\infty} \int_{2}^{\infty} \frac{z^{p}}{\xi^{p} (\xi + z)^{p} ( e^{z} + e^{-z} -2 \omega_{1} )^{k} } \, dz \, d \xi \,  d \omega \\ \leq C \ln^{p} \left( \frac{2}{\epsilon} \right)  \int_{\ln \left( \frac{2}{\epsilon} \right)}^{\infty} \frac{1}{\xi^{p}(\xi +2)^{p}} \, d \xi \leq C \ln^{p} \left( \frac{2}{\epsilon} \right)  \int_{\ln \left( \frac{2}{\epsilon} \right)}^{\infty} \frac{1}{\xi^{2p}} \, d \xi = \frac{C}{\ln^{p-1} \left( \frac{2}{\epsilon} \right)}.
\end{multline}
Also, using the Taylor series expansion of  $e^{z}+ e^{-z}$ at  $0$, we have
\begin{multline*}
    2  |\mathbb{S}^{k-1}| \ln^{p} \left( \frac{2}{\epsilon} \right) \int_{\mathbb{S}^{k-1}} \int_{\ln \left( \frac{2}{\epsilon} \right)}^{\infty} \int_{0}^{2} \frac{z^{p}}{\xi^{p} (\xi + z)^{p} ( e^{z} + e^{-z} -2 \omega_{1} )^{k} } \, dz \, d \xi \,  d \omega \\ \leq C \ln^{p} \left( \frac{2}{\epsilon} \right) \int_{\mathbb{S}^{k-1}} \int_{\ln \left( \frac{2}{\epsilon} \right)}^{\infty} \int_{0}^{2} \frac{z^{p}}{\xi^{p} (\xi + z)^{p} \left( 2 (1-\omega_{1}) + z^{2}  \right)^{k} } \, dz \, d \xi \,  d \omega.
\end{multline*}
Using  $1- \omega^{2}_{1} = |w'|^{2}$, and  $\frac{1}{1+ \omega_{1}} \geq \frac{1}{2}$, where  $\omega' = (\omega_{2}, \dots, \omega_{k})$, we obtain
\begin{equation*}
\begin{split}
    2  |\mathbb{S}^{k-1}| \ln^{p} \left( \frac{2}{\epsilon} \right) \int_{\mathbb{S}^{k-1}} & \int_{\ln \left( \frac{2}{\epsilon} \right)}^{\infty} \int_{0}^{2} \frac{z^{p}}{\xi^{p} (\xi + z)^{p} ( e^{z} + e^{-z} -2 \omega_{1} )^{k} } \, dz \, d \xi \,  d \omega \\ & \leq C  \ln^{p} \left( \frac{2}{\epsilon} \right) \int_{|\omega'| \leq 1} \int_{\ln \left( \frac{2}{\epsilon} \right)}^{\infty} \int_{0}^{2} \frac{z^{p}}{\xi^{p} (\xi + z)^{p} ( | \omega'|^{2} + z^{2} )^{k} } \, dz \, d \xi \, d \omega' \\ & \leq C \ln^{p} \left( \frac{2}{\epsilon} \right) \int_{\ln \left( \frac{2}{\epsilon} \right)}^{\infty} \frac{1}{\xi^{2p}} d \xi \int_{0}^{2} z^{p} \int_{|\omega'|<1} \frac{1}{(|w'|^{2}+ z^{2})^{k}} \, d \omega' \, dz.
    \end{split}
\end{equation*}
Using  \eqref{appen 1} (see Appendix  \ref{appendix}), we have  $\int_{0}^{2} z^{p} \int_{|\omega'|<1} \frac{1}{(|w'|^{2}+ z^{2})^{k}} \, d \omega' \, dz \leq C$. Therefore, we obtain
\begin{equation}\label{opt 3}
    2  |\mathbb{S}^{k-1}| \ln^{p} \left( \frac{2}{\epsilon} \right) \int_{\mathbb{S}^{k-1}} \int_{\ln \left( \frac{2}{\epsilon} \right)}^{\infty} \int_{0}^{2} \frac{z^{p}}{\xi^{p} (\xi + z)^{p} ( e^{z} + e^{-z} -2 \omega_{1} )^{k} } \, dz \, d\xi \,  d \omega  \leq \frac{C}{\ln^{p-1} \left( \frac{2}{\epsilon} \right)}.
\end{equation}
Combining  \eqref{opt 1},  \eqref{opt 2} and  \eqref{opt 3}, we have
\begin{equation*}
    J_{1} \leq \frac{C}{\ln^{p-1} \left( \frac{2}{\epsilon} \right)}.
\end{equation*}
Now, for  $J_{2}$, using the change of variable  $\eta = x_{k} - y_{k}$, we have
\begin{align*}
    J_{2} & = 2 \bigintsss_{\epsilon<|y_{k}|<1} \bigintsss_{0 < |x_{k}|< \epsilon} \frac{|u_{k}(x_{k}) - u_{k}(y_{k})|^{p}}{|x_{k} - y_{k}|^{2k}} \, dx_{k} \, dy_{k} \\ & = 2  \bigintsss_{0 < |x_{k}|< \epsilon} \frac{ \left| \ln \left( \frac{2}{|x_{k}|} \right) - \ln \left( \frac{2}{\epsilon} \right)  \right|^{p}}{ \ln^{p} \left( \frac{2}{|x_{k}|} \right)} \, dx_{k} \int_{\epsilon< |\eta -x_{k}| < 1} \frac{1}{|\eta|^{2k}}  \, d\eta.
\end{align*}
Since  $\{ \eta : \epsilon< |\eta -x_{k}| < 1  \} \subset \{ \eta : \epsilon - |x_{k}| < |\eta| < 1+ |x_{k}| \}$, it follows that, by applying this inclusion and a change of variable, we obtain
\begin{align*}
    J_{2} & \leq  2  \bigintsss_{0 < |x_{k}|< \epsilon}   \frac{ \left| \ln \left( \frac{2}{|x_{k}|} \right) - \ln \left( \frac{2}{\epsilon} \right)  \right|^{p}}{ \ln^{p} \left( \frac{2}{|x_{k|}} \right)}  \, dx_{k} \bigintsss_{\epsilon-|x_{k}|< |\eta| < 1+ |x_{k}|} \frac{1}{|\eta|^{2k}} \, d\eta \\ & = 2 |\mathbb{S}^{k-1}| \bigintsss_{0 < |x_{k}|< \epsilon} \frac{ \left| \ln \left( \frac{2}{|x_{k}|} \right) - \ln \left( \frac{2}{\epsilon} \right)  \right|^{p}}{ \ln^{p} \left( \frac{2}{|x_{k}|} \right)} \, dx_{k} \int_{\epsilon-|x_{k}|}^{1+ |x_k|} \frac{r^{k-1}}{r^{2k}} \,  dr \\ & \leq C \bigintsss_{0 < |x_{k}|< \epsilon} \frac{ \left| \ln \left( \frac{2}{|x_{k}|} \right) - \ln \left( \frac{2}{\epsilon} \right)  \right|^{p}}{ \ln^{p} \left( \frac{2}{|x_{k}|} \right)(\epsilon-|x_{k}|)^{k}} \, dx_{k}.
\end{align*}
Again, applying the change of variable in polar coordinates, we obtain
\begin{equation*}
    J_{2} \leq C \bigintsss_{0}^{\epsilon} \frac{ \left| \ln \left( \frac{2}{z} \right) - \ln \left( \frac{2}{\epsilon} \right)  \right|^{p}}{ \ln^{p} \left( \frac{2}{z} \right)(\epsilon-z)^{k}} z^{k-1} \,  dz.
\end{equation*}
Let  $t= \frac{\ln \left( \frac{2}{z} \right)}{ \ln \left( \frac{2}{\epsilon} \right)} - 1$. Then the above inequality will become
\begin{equation*}
    J_{2} \leq C \ln \left( \frac{2}{\epsilon} \right) \bigintsss_{0}^{\infty} \frac{t^{p}}{(t+1)^{p} \left( e^{t \ln \left( \frac{2}{\epsilon} \right)} - 1 \right)^{k}} \, dt \leq C \ln \left( \frac{2}{\epsilon} \right) \bigintsss_{0}^{\infty} \frac{t^{p}}{ \left( e^{t \ln \left( \frac{2}{\epsilon} \right)} - 1 \right)^{k}} \, dt.
\end{equation*}
Applying the change of variable  $q= t \ln \left( \frac{2}{\epsilon} \right)$ and using the fact that $\int_{0}^{\infty} \frac{t^{q}}{\left( e^{q} -1\right)^{k}} \, dq < \infty$, we get
\begin{equation*}
    J_{2} \leq \frac{C}{\ln^{p} \left( \frac{2}{\epsilon} \right)} \int_{0}^{\infty} \frac{q^{p}}{ \left( e^{q} -1 \right)^{k}} \, dq \leq \frac{C}{\ln^{p} \left( \frac{2}{\epsilon} \right)}.
\end{equation*}
Combining  $J_{1}$ and  $J_{2}$ in  \eqref{opt eq gsk}, we obtain
\begin{equation*}
     [u_{k}]^{p}_{W^{s,p}(B^{k}_{1}(0))} \leq \frac{C}{\ln^{p-1} \left( \frac{2}{\epsilon} \right)}.
\end{equation*}
This finishes the proof of lemma.
\end{proof}

\smallskip

We now define a sequence of functions using  $u_{k}$ that belongs to  $W^{s,p}_{0}(\Omega \setminus K)$, thereby ensuring that Theorem  \ref{theorem 1 sp=k} holds true for this functions. The next lemma establishes a sequence of functions $\{ \Phi_{\mu} \}_{\mu, \epsilon}$ in $W^{s,p}_{0}(\Omega \setminus K)$, where $\Omega = B^{k}_{1}(0) \times (0,1)^{d-k}$ with  $ K= \{ x= (x_{k}, x_{d-k}) \in \mathbb{R}^{d} : x_{k} = 0  \ \text{and} \ x_{d-k} \in [0,1]^{d-k} \}$ which converges to a constant function $1$ in $W^{s,p}_{0}(\Omega \setminus K )$ as $\mu \to 0$ and then $\epsilon \to 0$.

\begin{lemma}\label{lemma for Phi mu}
    Let $1<k<d$ with $k \in \mathbb{N}$ and $sp=k$, and let  $\Omega = B^{k}_{1}(0) \times (0,1)^{d-k}$ with  $ K= \{ x= (x_{k}, x_{d-k}) \in \mathbb{R}^{d} : x_{k} = 0  \ \text{and} \ x_{d-k} \in [0,1]^{d-k} \}$. Let   $  0< \mu< \epsilon <1$  and $\Phi_{\mu} \in W^{s,p}_{0}(\Omega \setminus K)$ defined as 
\begin{equation}\label{fn phi_mu}
    \Phi_{\mu}(x_{k}, x_{d-k}) = \begin{cases}
    0, & 0 \leq |x_{k}| \leq \mu \\
    \left( \frac{\ln \left( \frac{2}{\epsilon} \right)}{ \ln \left( \frac{2}{|x_{k}|} \right)} - \frac{\ln \left( \frac{2}{\epsilon} \right)}{ \ln \left( \frac{2}{\mu} \right)} \right) \frac{1}{1- \frac{\ln \left( \frac{2}{\epsilon} \right) }{ \ln \left( \frac{2}{\mu} \right) }}, & \mu< |x_{k}| < \epsilon \\
        1, &  \epsilon \leq |x_{k}| < 1  . 
    \end{cases}
\end{equation}
Then $\Phi_{\mu}$ converges to a constant function $1$ in $W^{s,p}_{0}(\Omega \setminus K)$ as $\mu \to 0$ and then $\epsilon \to 0$.
\end{lemma}
\begin{proof}
Since, by the dominated convergence theorem,  $\Phi_{\mu} \to 1$ as  $\mu \to 0$ and then  $\epsilon \to 0$ in  $L^{p}(\Omega)$. Therefore, it is sufficient to prove $[\Phi_{\mu}]^{p}_{W^{s,p}(\Omega)} \to 0$ as $\mu \to 0$ and then $\epsilon \to 0$. Let us calculate the Gagliardo seminorm of $\Phi_{\mu}$ and using slicing lemma (see  \eqref{slicing lemma}, Appendix  \ref{appendix}) with $sp=k$, we have
    \begin{equation}\label{eqn on lemma Phi mu}
    \begin{split}
        [\Phi_{\mu}]^{p}_{W^{s,p}(\Omega)} & \leq  C \int_{B^{k}_{1}(0)} \int_{B^{k}_{1}(0)} \frac{|\Phi_{\mu}(x_{k},x_{d-k})- \Phi_{\mu}(y_{k},y_{d-k})|^{p}}{|x_{k}-y_{k}|^{2k}} \, dx_{k} \, dy_{k} \\ & = C \Bigg( \int_{\mu < |y_{k}|< \epsilon} \int_{\mu < |x_{k}|< \epsilon} \frac{|\Phi_{\mu}(x_{k},x_{d-k})- \Phi_{\mu}(y_{k},y_{d-k})|^{p}}{|x_{k}-y_{k}|^{2k}} \, dx_{k} \, dy_{k} \\ & \quad  + 2  \int_{0 < |y_{k}|< \mu} \int_{\mu < |x_{k}|< \epsilon} \frac{|\Phi_{\mu}(x_{k},x_{d-k})- \Phi_{\mu}(y_{k},y_{d-k})|^{p}}{|x_{k}-y_{k}|^{2k}} \, dx_{k} \, dy_{k} \\ & \quad + 2  \int_{0 < |y_{k}|< \mu} \int_{\epsilon < |x_{k}|< 1} \frac{|\Phi_{\mu}(x_{k},x_{d-k})- \Phi_{\mu}(y_{k},y_{d-k})|^{p}}{|x_{k}-y_{k}|^{2k}} \, dx_{k} \, dy_{k} \\ & \quad + 2  \int_{\mu < |y_{k}|< \epsilon} \int_{\epsilon < |x_{k}|< 1} \frac{|\Phi_{\mu}(x_{k},x_{d-k})- \Phi_{\mu}(y_{k},y_{d-k})|^{p}}{|x_{k}-y_{k}|^{2k}} \, dx_{k} \, dy_{k} \Bigg)  \\ & = : L_{1}+ L_{2} + L_{3} + L_{4}.
    \end{split}
    \end{equation}
    For $L_{1}$, given $\mu < |x_{k}|< \epsilon$ and $\mu < |y_{k}| < \epsilon$, we have
    \begin{equation*}
        |\Phi_{\mu}(x_{k},x_{d-k})- \Phi_{\mu}(y_{k},y_{d-k})| = \left( \frac{\ln \left( \frac{2}{\mu} \right)}{\ln \left( \frac{2}{\mu} \right) - \ln \left( \frac{2}{\epsilon} \right) } \right) |u_{k}(x_{k}) - u_{k}(y_{k})|,
    \end{equation*}
    where $u_{k}$ as defined in  \eqref{fn u_k}. Therefore, we have
    \begin{align*}
        L_{1} & = \left( \frac{\ln \left( \frac{2}{\mu} \right)}{\ln \left( \frac{2}{\mu} \right) - \ln \left( \frac{2}{\epsilon} \right) } \right)^{p} \int_{\mu < |x_{k}|< \epsilon} \int_{\mu < |y_{k}|< \epsilon} \frac{|u_{k}(x_{k})- u_{k}(y_{k})|^{p}}{|x_{k}-y_{k}|^{2k}} \, dx_{k} \,  dy_{k} \\ & \leq \left( \frac{\ln \left( \frac{2}{\mu} \right)}{\ln \left( \frac{2}{\mu} \right) - \ln \left( \frac{2}{\epsilon} \right) } \right)^{p} \int_{0 < |x_{k}|< \epsilon} \int_{0 < |y_{k}|< \epsilon} \frac{|u_{k}(x_{k})- u_{k}(y_{k})|^{p}}{|x_{k}-y_{k}|^{2k}} \, dx_{k} \, dy_{k} .
    \end{align*}
    For $L_{2}$,  given $0 < |y_{k}|< \mu$ and $\mu < |x_{k}| < \epsilon$, we have $ \ln \left( \frac{2}{\mu} \right) < \ln \left( \frac{2}{|y_{k}|} \right)$. Therefore, we obtain
    \begin{equation*}
    \begin{split}
        |\Phi_{\mu}(x_{k},x_{d-k})- \Phi_{\mu}(y_{k},y_{d-k})| & =  \left( \frac{\ln \left( \frac{2}{\mu} \right)}{\ln \left( \frac{2}{\mu} \right) - \ln \left( \frac{2}{\epsilon} \right)} \right)   \left|  \frac{\ln \left( \frac{2}{\epsilon} \right)}{ \ln \left( \frac{2}{|x_{k}|} \right)} - \frac{\ln \left( \frac{2}{\epsilon} \right)}{ \ln \left( \frac{2}{\mu} \right)} \right| \\ & \leq  \left( \frac{\ln \left( \frac{2}{\mu} \right)}{\ln \left( \frac{2}{\mu} \right) - \ln \left( \frac{2}{\epsilon} \right)} \right) \left|  \frac{\ln \left( \frac{2}{\epsilon} \right)}{ \ln \left( \frac{2}{|x_{k}|} \right)} - \frac{\ln \left( \frac{2}{\epsilon} \right)}{ \ln \left( \frac{2}{|y_{k}|} \right)} \right| \\ & = \left( \frac{\ln \left( \frac{2}{\mu} \right)}{\ln \left( \frac{2}{\mu} \right) - \ln \left( \frac{2}{\epsilon} \right)} \right) |u_{k}(x_{k})- u_{k}(y_{k})|,
    \end{split}
    \end{equation*}
    where $u_{k}$ as defined in  \eqref{fn u_k}. Therefore, we have
    \begin{align*}
        L_{2} & \leq \left( \frac{\ln \left( \frac{2}{\mu} \right)}{\ln \left( \frac{2}{\mu} \right) - \ln \left( \frac{2}{\epsilon} \right) } \right)^{p} \int_{0 < |y_{k}|< \mu} \int_{\mu < |x_{k}|< \epsilon} \frac{|u_{k}(x_{k})- u_{k}(y_{k})|^{p}}{|x_{k}-y_{k}|^{2k}} \, dx_{k} \, dy_{k} \\ & \leq \left( \frac{\ln \left( \frac{2}{\mu} \right)}{\ln \left( \frac{2}{\mu} \right) - \ln \left( \frac{2}{\epsilon} \right) } \right)^{p} \int_{0 < |y_{k}|< \epsilon} \int_{0 < |x_{k}|< \epsilon} \frac{|u_{k}(x_{k})- u_{k}(y_{k})|^{p}}{|x_{k}-y_{k}|^{2k}} \, dx_{k} \, dy_{k} .
    \end{align*}
    Now for $L_{3}$. Let $0< |y_{k}|< \mu$ and choose $C>0$ such that $\frac{1}{C} \leq 1- \frac{\ln \left( \frac{2}{\epsilon} \right)}{ \ln \left( \frac{2}{|y_{k}|} \right) }$. Since, as $\mu \to 0$, we have $\frac{\ln \left( \frac{2}{\epsilon} \right)}{ \ln \left( \frac{2}{|y_{k}|} \right)} \to 0$, it follows that for sufficiently small $\mu>0$, we have $\frac{1}{2} \leq  \frac{1}{C} \leq 1- \frac{\ln \left( \frac{2}{\epsilon} \right)}{ \ln \left( \frac{2}{|y_{k}|} \right) }$. Therefore, for any $0< |y_{k}|< \mu$ and $\epsilon < |x_{k}| < 1$ with sufficiently small $\mu$, we have
    \begin{equation*}
        |\Phi_{\mu}(x_{k},x_{d-k})- \Phi_{\mu}(y_{k},y_{d-k})| =1 \leq 2 \left| 1- \frac{\ln \left( \frac{2}{\epsilon} \right)}{ \ln \left( \frac{2}{|y_{k}|} \right) }  \right| = 2 |u_{k}(x_{k})- u_{k}(y_{k})|,
    \end{equation*}
  where $u_{k}$ as defined in  \eqref{fn u_k}.   Therefore, for sufficiently small $\mu>0$, we have
    \begin{align*}
        L_{3} & \leq 2^{p} \int_{0 < |y_{k}|< \mu} \int_{\epsilon < |x_{k}|< 1} \frac{|u_{k}(x_{k})- u_{k}(y_{k})|^{p}}{|x_{k}-y_{k}|^{2k}} \, dx_{k} \, dy_{k} \\ & \leq 2^{p} \int_{0 < |y_{k}|< \epsilon} \int_{\epsilon < |x_{k}|< 1} \frac{|u_{k}(x_{k})- u_{k}(y_{k})|^{p}}{|x_{k}-y_{k}|^{2k}} \, dx_{k} \, dy_{k} .
    \end{align*}
  Finally, for $L_{4}$. Since for any $\mu< |y_{k}| < \epsilon$, it follows that
  \begin{equation*}
      1 - \left( \frac{\ln \left( \frac{2}{\epsilon} \right)}{ \ln \left( \frac{2}{|y_{k}|} \right)} - \frac{\ln \left( \frac{2}{\epsilon} \right)}{ \ln \left( \frac{2}{\mu} \right)} \right) \frac{1}{1- \frac{\ln \left( \frac{2}{\epsilon} \right) }{ \ln \left( \frac{2}{\mu} \right) }} \to 1- \frac{\ln \left( \frac{2}{\epsilon}  \right)}{ \ln \left( \frac{2}{|y_{k}|} \right)} \hspace{3mm} \text{as} \hspace{3mm} \mu \to 0.
  \end{equation*}
  Therefore, for sufficiently small $\mu>0$, given $\mu< |y_{k}| < \epsilon$ and $\epsilon < |x_{k}| <1$, we have
  \begin{align*}
      |\Phi_{\mu}(x_{k},x_{d-k})- \Phi_{\mu}(y_{k},y_{d-k})| & = \left| 1 - \left( \frac{\ln \left( \frac{2}{\epsilon} \right)}{ \ln \left( \frac{2}{|y_{k}|} \right)} - \frac{\ln \left( \frac{2}{\epsilon} \right)}{ \ln \left( \frac{2}{\mu} \right)} \right) \frac{1}{1- \frac{\ln \left( \frac{2}{\epsilon} \right) }{ \ln \left( \frac{2}{\mu} \right) }} \right| \\ & \leq 5 \left| 1- \frac{\ln \left( \frac{2}{\epsilon}  \right)}{ \ln \left( \frac{2}{|y_{k}|} \right)} \right| = 5 | u_{k}(x_{k})-u_{k}(y_{k})|.
  \end{align*}
  Therefore, for sufficiently small $\mu >0$, we have
  \begin{align*}
      L_{4} & \leq 5^{p} \int_{\mu < |y_{k}|< \epsilon} \int_{\epsilon < |x_{k}|< 1} \frac{|u_{k}(x_{k})- u_{k}(y_{k})|^{p}}{|x_{k}-y_{k}|^{2k}} \, dx_{k} \, dy_{k} \\ & \leq 5^{p} \int_{0 < |y_{k}|< \epsilon} \int_{\epsilon < |x_{k}|< 1} \frac{|u_{k}(x_{k})- u_{k}(y_{k})|^{p}}{|x_{k}-y_{k}|^{2k}} \, dx_{k} \, dy_{k} .
  \end{align*}
  Combining  $L_{1}, ~ L_{2}, ~ L_{3}$ and  $L_{4}$ in  \eqref{eqn on lemma Phi mu}, taking limit $\mu \to 0$ and using Lemma  \ref{lemma on u_k}, we obtain for some suitable constant  $C>0$ does not depends on  $\mu$ and  $\epsilon$,
  \begin{equation*}
      \lim_{\mu \to 0}  [\Phi_{\mu}]^{p}_{W^{s,p}(\Omega)} \leq C[u_{k}]^{p}_{W^{s,p}(B^{k}_{1}(0))} \leq \frac{C}{ \ln^{p-1} \left( \frac{2}{\epsilon} \right)}.
  \end{equation*}
Now taking $\epsilon \to 0$ in the above inequality, we obtain $[\Phi_{\mu}]^{p}_{W^{s,p}(\Omega)} \to 0$. This proves the lemma.
\end{proof}

Dyda and Kijaczko in  \cite[Lemma 13]{dyda2022} proved that  $W^{s,p}(\Omega)= W^{s,p}_{0}(\Omega)$ if and only if the constant function  $1 \in W^{s,p}_{0}(\Omega)$. Using similar techniques illustrated in  \cite[Lemma 13]{dyda2022}, we will prove the following lemma for the case  $sp=k$.

\begin{lemma}\label{density lemma}
    Let  $\Omega$ be a bounded domain in $\mathbb{R}^{d}$, where $d \geq 3$ and let  $K \subset \Omega$ be a compact set of codimension $k$ of class  $C^{0,1}$, where $1<k<d, ~k \in \mathbb{N}$ and $sp=k$. Then $ W^{s,p}_{0}(\Omega \setminus K) = W^{s,p}(\Omega)$.
\end{lemma}
\begin{proof}
    It is sufficient to prove for flat boundary case. For general domain  $\Omega$ and a compact set  $K \subset \Omega$ of codimension  $k$ of class  $C^{0,1}$, where $1<k<d, ~ k \in \mathbb{N}$ follows from patching business. Let  $\Omega = B^{k}_{1}(0) \times (0,1)^{d-k}$ and  $K= \{ x= (x_{k}, x_{d-k}) \in \mathbb{R}^{d} : x_{k} = 0  \ \text{and} \ x_{d-k} \in [0,1]^{d-k} \}$. Consider a function  $\Phi_{\mu}$ defined in  \eqref{fn phi_mu}. Then  $\Phi_{\mu} \in W^{s,p}_{0}(\Omega \setminus K)$ and using Lemma \ref{lemma for Phi mu}, we have the sequence of functions  $\{ \Phi_{\mu} \}_{\epsilon,  \mu}$ in  $W^{s,p}_{0}(\Omega \setminus K)$ converges to a constant function  $1$ and  $W^{s,p}_{0}(\Omega \setminus K)$ is a closed subspace of  $W^{s,p}(\Omega)$ which imply  $1 \in W^{s,p}_{0}(\Omega \setminus K)$. Let  $u \in W^{s,p}(\Omega)$ and we may also assume  $u \in L^{\infty}(\Omega)$. Define a sequence of functions  $\{h_{\mu}=u \Phi_{\mu} \}_{\epsilon,  \mu}$. Then  $h_{\mu} \in W^{s,p}_{0}(\Omega \setminus K)$ and clearly  $h_{\mu} \to u$ in  $L^{p}(\Omega)$ as  $\mu \to 0$ and then  $\epsilon \to 0$. Now let us calculate the Gagliardo seminorm of  $h_{\mu}-u$,
    \begin{equation*}
    \begin{split}
        [h_{\mu} - u]^{p}_{W^{s,p}(\Omega)} & = \int_{\Omega} \int_{\Omega} \frac{|u(x)(1-\Phi_{\mu}(x))-u(y)(1- \Phi_{\mu}(y))|^{p}}{|x-y|^{d+sp}} \, dx \, dy \\ & \leq 2^{p-1} \int_{\Omega} \int_{\Omega}  \frac{|u(x)|^{p}|\Phi_{\mu}(x)-\Phi_{\mu}(y)|^{p}}{|x-y|^{d+sp}} \, dx \, dy \\ & \quad + 2^{p-1} \int_{\Omega} \int_{\Omega}  \frac{|1-\Phi_{\mu}(x)|^{p}|u(x)-u(y)|^{p}}{|x-y|^{d+sp}} \, dx \, dy \\ & \leq  2^{p-1}\|u\|^{p}_{\infty} [\Phi_{\mu}]^{p}_{W^{s,p}(\Omega)} + 2^{p-1} \int_{\Omega} \int_{\Omega}  \frac{|1-\Phi_{\mu}(x)|^{p}|u(x)-u(y)|^{p}}{|x-y|^{d+sp}} \, dx \, dy.
    \end{split}
    \end{equation*} 
    Since $\Phi_{\mu} \to 1$ in $L^{p}(\Omega)$ as $\mu \to 0$ and then $\epsilon \to 0$, there exists a subsequence (which we again denote by $\{\Phi_{\mu}\}$) such that $\Phi_{\mu} \to 1$ almost everywhere as $\mu \to 0$ and then $\epsilon \to 0$. Therefore, using the dominated convergence theorem and the fact that   $[\Phi_{\mu}]^{p}_{W^{s,p}(\Omega)} \to 0$ as  $\mu \to 0$ and then  $\epsilon \to 0$, we obtain
    \begin{equation*}
        [h_{\mu} - u]^{p}_{W^{s,p}(\Omega)} \to 0 \hspace{3mm} \text{as} \hspace{3mm} \mu \to 0 \ \text{and then} \ \epsilon \to 0. 
    \end{equation*}
    Hence,  $u \in W^{s,p}_{0}(\Omega \setminus K)$. This finishes the proof of lemma.
\end{proof}

\begin{rem}\label{rem sp<k}
    Assume  $\Omega$ is a bounded Lipschitz domain in $\mathbb{R}^{d}$, where $d \geq 2$,  $sp=k$ and  $s'<s$, it follows that  $s'p<k$. We have the embedding  $W^{s,p}(\Omega) \subset W^{s',p}(\Omega)$ and the inequality  $\|u\|_{W^{s',p}(\Omega)} \leq C\|u\|_{W^{s,p}(\Omega)}$ (see \cite[Proposition 2.1]{di2012hitchhikers}). Therefore, following a similar approach illustrated as above lemma, we can deduce that  $W^{s',p}_{0}(\Omega \setminus K) = W^{s',p}(\Omega)$. Moreover, if we assume  $p>k$. Then   $W^{s,p}_{0}(\Omega \setminus K) = W^{s,p}(\Omega)$ for  $sp<k$.
\end{rem}

\smallskip

The following lemma proves the optimality of the weight function in Theorem  \ref{theorem 1 sp=k} for general bounded Lipschitz domain  $\Omega$ with a compact set  $K \subset \Omega$ of codimension  $k$ of class  $C^{0,1}$, where  $1<k<d, ~ k \in \mathbb{N}$ for  $\tau =p$. The above three lemmas, Lemma  \ref{lemma on u_k}, Lemma  \ref{lemma for Phi mu} and Lemma  \ref{density lemma} are the key ingredients to prove the next lemma.

\begin{lemma}\label{lemma proof of opt gen domain}
    Let  $\Omega$ be a bounded Lipschitz domain in  $\mathbb{R}^{d}$, where $d \geq 2$, and let  $K \subset \Omega$ be a compact set of codimension  $k$ of class  $C^{0,1}$,  where $1<k<d, ~ k \in \mathbb{N}$ and  $sp=k$. Then the logarithmic weight function in Theorem  \ref{theorem 1 sp=k} is optimal for  $\tau =p$.
\end{lemma}
\begin{proof}
We prove the optimality of weight function in Lemma  \ref{flat case sp=k 1} with  $\Omega= B^{k}_{1}(0) \times (0,1)^{d-k}$ and  $K= \{ x= (x_{k}, x_{d-k}) \in \mathbb{R}^{d} : x_{k} = 0  \ \text{and} \ x_{d-k} \in [0,1]^{d-k} \}$. For general domain, it follows from patching business. Assume that the weight function in Lemma  \ref{flat case sp=k 1} is not optimal. Suppose there exists a function  $f$ with the property  $f(x_{k}) \to \infty$ as $|x_{k}| \to 0$. This function  $f$ allows for a further improvement of the inequality in Theorem  \ref{theorem 1 sp=k} such that the following holds:
\begin{equation*}
    \bigintsss_{\Omega} \frac{|f(x_{k})||u(x) - (u)_{\Omega}|^{p}}{|x_{k}|^{k} \ln^{p} \left( \frac{2}{|x_{k}|}  \right)} \, dx  \leq [u]^{p}_{W^{s,p}(\Omega)}, \quad \forall \ u \in W^{s,p}_{0}(\Omega \setminus K).
\end{equation*}
Let  $v_{\epsilon}$ be a function defined on  $\Omega= B^{k}_{1}(0) \times (0,1)^{d-k}$ such that  $v_{\epsilon}(x_{k}, x_{d-k}) = u_{k}(x_{k})$ for all  $x= (x_{k}, x_{d-k}) \in \Omega, ~ x_{k} \in B^{k}_{1}(0)$ and  $x_{d-k} \in (0,1)^{d-k}$, where the functions  $u_{k}$ defined as in  \eqref{fn u_k}. From Lemma  \ref{density lemma} and slicing inequality (see, Subsection  \ref{slicing lemma}), we have  $v_{\epsilon} \in W^{s,p}_{0}(\Omega \setminus K)$ for the case  $sp=k$. Therefore, the above inequality hold true for  $v_{\epsilon}$. Using a similar type of computation as done in Lemma  \ref{lemma on u_k}, we have
\begin{align*}
  \bigintsss_{\Omega} \frac{|f(x_{k})||v_{\epsilon}(x_{k},x_{d-k}) - (v_{\epsilon})_{\Omega}|^{p}}{|x_{k}|^{k} \ln^{p} \left( \frac{2}{|x_{k}|}  \right)} \, dx & =  \bigintsss_{B^{k}_{1}(0)} \frac{|f(x_{k})||u_{k}(x_{k}) - (u_{k})_{B^{k}_{1}(0)}|^{p}}{|x_{k}|^{k} \ln^{p} \left( \frac{2}{|x_{k}|}  \right)} \, dx_{k} \\ & \geq |f(x_{k,\epsilon})| \left( \frac{C}{\ln^{p -1} \left( \frac{2}{\epsilon} \right)} + O(\epsilon) \right),
\end{align*}
where  $|x_{k,\epsilon}| = \epsilon$ and $x_{k, \epsilon} \in \mathbb{R}^{k}$. Using the above two inequalities, slicing inequality (see Subsection  \ref{slicing lemma}) and Lemma  \ref{lemma on u_k}, we have
\begin{align*}
|f(x_{k,\epsilon})| \left( \frac{C}{\ln^{p -1} \left( \frac{2}{\epsilon} \right)} + O(\epsilon) \right) & \leq   \bigintsss_{\Omega} \frac{|f(x_{k})||v_{\epsilon}(x_{k}, x_{d-k}) - (v_{\epsilon})_{\Omega}|^{p}}{|x_{k}|^{k} \ln^{p} \left( \frac{2}{|x_{k}|}  \right)} \, dx \\ & \leq C [v_{\epsilon}]^{p}_{W^{s,p}(\Omega)}  \leq C [u_{k}]^{p}_{W^{s,p}(B^{k}_{1}(0))} \leq \frac{C}{\ln^{p-1} \left( \frac{2}{\epsilon} \right)}.
\end{align*}
Therefore, we have for some constant  $C>0$ does not depend on $\epsilon$
\begin{equation*}
    |f(x_{k,\epsilon})| \left( C + \left( \ln^{p-1} \frac{2}{\epsilon} \right)  O(\epsilon) \right) \leq C 
\end{equation*}
which is a contradictions as  $|f(x_{k,\epsilon})| \to \infty$ as  $|x_{k, \epsilon}|=\epsilon \to 0$. This proves the optimality of weight function in Lemma  \ref{flat case sp=k 1}.
\end{proof}


\section{Appendix}\label{appendix}

In this section, we will prove some estimates which we have used to prove the main results. This estimates are obvious and easy to establish.  

\subsection{Some inequality}\label{some ineq} \textbf{(1)} We establish the following inequality holds
\begin{align*}
    \int_{0}^{2} z^{p} \int_{|\omega'|<1} \frac{1}{(|w'|^{2}+ z^{2})^{k}} \, d \omega' \, dz & = \int_{0}^{2} z^{p} \int_{0}^{1} \frac{\theta^{k-2}}{z^{2k} \left( \frac{\theta^{2}}{z^{2}} +1 \right)^{k}} \, d \theta \, dz \\ & = \int_{0}^{2} z^{p} \int_{0}^{1} \frac{\left( \frac{\theta}{z} \right)^{k-2} z^{k-2} }{z^{2k} \left( \frac{\theta^{2}}{z^{2}} +1 \right)^{k}} \, d \theta \, dz.
\end{align*}
Using the change of variable  $\frac{\theta}{z} = t$ and  $sp=k$, i.e.,  $p>k$, we obtain
\begin{equation}\label{appen 1}
    \int_{0}^{2} z^{p-k-1} \int_{0}^{\frac{1}{z}} \frac{t^{k-2}}{(t^{2}+1)^{k}} \, dt \, dz \leq \int_{0}^{2} z^{p-k-1} \int_{0}^{\infty} \frac{t^{k-2}}{(t^{2}+1)^{k}} \, dt \, dz \leq C.
\end{equation}

\smallskip

\textbf{(2)} We will calculate the inequality which is useful in establishing our main results for flat boundary case. Recall that $$\Omega_{n} = \left\{ x= (x_{k}, x_{d-k}) :  |x_{k}| < 1 \ \text{and} \ x_{d-k} \in (-n, n)^{d-k}   \right\}$$ and the sets $A_{\ell}$ are defined in Section  \ref{The flat boundary case}. We aim to show that
\begin{equation}\label{se1}
     \sum_{\ell= m}^{-2} [u]^{p}_{W^{s,p}(A_{\ell} \cup A_{\ell+1})} \leq 2 [u]^{p}_{W^{s,p}(\Omega_{n})}.
\end{equation}
Consider two families of sets:
\begin{equation*}
  \mathcal{E}:=  \left\{ A_{\ell} \cup A_{\ell+1} : -\ell \ \text{is even and} \ \ell \leq -1  \right\}
\end{equation*}
and
\begin{equation*}
  \mathcal{O}:=  \left\{ A_{\ell} \cup A_{\ell+1} : -\ell \ \text{is odd and} \ \ell \leq -1  \right\}.
\end{equation*}
Then  $\mathcal{E}$ and  $\mathcal{O}$ are collection of mutually disjoint sets respectively. Define
\begin{equation*}
    \mathcal{F}_{e} := \bigcup_{\substack{\ell =m \\ -\ell \ \text{is even}}}^{-2} A_{\ell} \cup A_{\ell+1} \hspace{3mm} \text{and} \hspace{3mm} \mathcal{F}_{o} := \bigcup_{\substack{\ell =m \\ -\ell \ \text{is odd}}}^{-2} A_{\ell} \cup A_{\ell+1}.
\end{equation*} 
From the definition of  $A_{\ell}$, we have $\mathcal{F}_{e} \subset \Omega_{n}$ and $\mathcal{F}_{o} \subset \Omega_{n}$. Therefore, we have
\begin{align}\label{sum Ak A(k+1)}
    \sum_{\ell= m}^{-2} [u]^{p}_{W^{s,p}(A_{\ell} \cup A_{\ell+1})} & = \sum_{\substack{\ell =m \\ -\ell \ \text{is even}}}^{-2} [u]^{p}_{W^{s,p}(A_{\ell} \cup A_{\ell+1})} + \sum_{\substack{\ell =m \\ -\ell \ \text{is odd}}}^{-2} [u]^{p}_{W^{s,p}(A_{\ell} \cup A_{\ell+1})} \nonumber \\ & \leq [u]^{p}_{W^{s,p}(\mathcal{F}_{e})} + [u]^{p}_{W^{s,p}(\mathcal{F}_{o})} \leq 2 [u]^{p}_{W^{s,p}(\Omega_{n})}.
\end{align}
This establishes the desired inequality.

\subsection{Slicing inequality}\label{slicing lemma}
We introduce a slicing lemma for the Gagliardo seminorm. While this inequality can be found in  \cite[Section 6.2]{leonibook}, for the sake of completeness and self-containment, we present its proof within this article. Let  $v_{\epsilon}$ be a function defined on  $\Omega= B^{k}_{1}(0) \times (0,1)^{d-k}$ such that  $v_{\epsilon}(x_{k}, x_{d-k}) = u_{k}(x_{k})$ for all  $x= (x_{k}, x_{d-k}) \in \Omega, ~ x_{k} \in B^{k}_{1}(0)$ and  $x_{d-k} \in (0,1)^{d-k}$, where the functions  $u_{k}$ defined as in  \eqref{fn u_k} (Section  \ref{optimality section}). Therefore, we have
\begin{align*}
    [v_{\epsilon}]^{p}_{W^{s,p}(\Omega)} & =  \int_{\Omega} \int_{\Omega} \frac{|u_{k}(x_{k})- u_{k}(x_{k})|^{p}}{|x-y|^{d+sp}} \, dx \, dy \\ & =  \bigintss_{\Omega} \bigintss_{\Omega} \frac{|u_{k}(x_{k})- u_{k}(x_{k})|^{p}}{|x_{k}-y_{k}|^{d+sp} \left( 1+ \sum_{i=k+1}^{d} \left| \frac{x_{i}-y_{i}}{|x_{k}-y_{k}|} \right|^{2} \right)^{\frac{d+sp}{2}}} \, dx \,  dy .
\end{align*}
Using the change of variables  $z_{i} =\frac{x_{i}- y_{i}}{|x_{k}-y_{k}|}$ and using
\begin{equation*}
    \int_{(0, \infty)^{d-k}} \frac{1}{(1+ \sum_{i=k+1}^{d} |z_{i}|^{2} )^{\frac{d+sp}{2}}}  \, dz_{k+1} \dots \, dz_{d} < \infty,
\end{equation*}
we obtain
\begin{equation*}
    [v_{\epsilon}]^{p}_{W^{s,p}(\Omega)} \leq C [u_{k}]^{p}_{W^{s,p}(B^{k}_{1}(0))}.
\end{equation*}

\bigskip

\bigskip 

\textbf{Acknowledgement:} We express our gratitude to the Department of Mathematics and Statistics at the Indian Institute of Technology Kanpur for providing conductive research environment. For this work, Adimurthi acknowledges support from IIT Kanpur, while P. Roy is supported by the Core Research Grant (CRG/2022/007867) of SERB. V. Sahu is grateful for the support received through MHRD, Government of India (GATE fellowship).

\end{document}